\documentclass[reqno, 10pt, oneside]{amsart}

\usepackage{latexsym, amsmath,amssymb}
\usepackage{graphicx}
\usepackage[hidelinks]{hyperref}

\usepackage[initials, alphabetic,backrefs,msc-links]{amsrefs}

\usepackage{xcolor}
\usepackage{enumitem}
\DeclareMathOperator*{\argmin}{arg\,min}
\usepackage[utf8]{inputenc} 
\newcommand{\cal}{\mathcal}

 \numberwithin{equation}{section}

\newcommand{\hcal}{\mathcal{H}}

\usepackage{amsmath}

\def\XXint#1#2#3{{\setbox0=\hbox{$#1{#2#3}{%
\int}$ }
\vcenter{\hbox{$#2#3$ }}\kern-.6\wd0}}

\renewcommand{\epsilon}{\varepsilon}
\newtheorem{theorem}{Theorem}

\newtheorem{lemma}[theorem]{Lemma}
\newtheorem{corollary}[theorem]{Corollary}

\newtheorem{proposition}[theorem]{Proposition}
\newtheorem{definition}[theorem]{Definition}
\newtheorem{remark}[theorem]{Remark}
\newcommand{\bth}{\begin{theorem}}
\newcommand{\ble}{\begin{lemma}}
\newcommand{\bcor}{\begin{corr}}

\newcommand{\bdeff}{\begin{deff}}

\newcommand{\bprop}{\begin{proposition}}
\newcommand{\ele}{\end{lemma}}
\newcommand{\ecor}{\end{corr}}
\newcommand{\edeff}{\end{deff}}

\numberwithin{theorem}{section}

\newcommand{\eprop}{\end{proposition}}

\newcommand{\supp}{\text{supp }}
\renewcommand{\Pi}{\varPi}

\renewcommand{\epsilon}{\varepsilon}
\newcommand{\sgn}{{\text {sgn}}}

\newcommand{\R}{{\mathbb R}}
\newcommand{\Z}{{\mathbb Z}}

\usepackage{esint}

\usepackage{comment}

\newcommand{\ed}{\color{black}}

\usepackage{geometry}

\allowdisplaybreaks

\makeatletter
\@namedef{subjclassname@2020}{\textup{2020} Mathematics Subject Classification}
\makeatother

\begin{document}

\title[Commutators of Fractional Integrals with $\operatorname{BMO}^\beta$ Functions]
{Commutators of Fractional Integrals with $\operatorname{BMO}^\beta$ Functions}

\author[Y.-W.B. Chen]{You-Wei Benson Chen}

\address[Y.-W. Chen]{
Department of Mathematics, National Changhua University of Education,
No. 1, Jin-De Road, Changhua City, Taiwan
}
\email{bensonchen@cc.ncue.edu.tw}

\author[A. Claros]{Alejandro Claros}

\address[A. Claros]{BCAM -- Basque Center for Applied Mathematics, Bilbao, Spain}
\email{aclaros@bcamath.org}

\address[A. Claros]{Universidad del País Vasco / Euskal Herriko Unibertsitatea (UPV/EHU), Bilbao, Spain}
\email{aclaros003@ikasle.ehu.eus}

\thanks{Y.-W. Chen is supported by the National Science and Technology Council of Taiwan under research grant number 114-2115-M-018-003-MY2. A. Claros is partially supported by the Basque Government through the BERC 2022-2025 program, by the Spanish Ministry of Science and Innovation through Grant PRE2021-099091 funded by BCAM Severo Ochoa accreditation CEX2021-001142-S/MICIN/AEI/10.13039/501100011033 and by ESF+, and by the project PID2023-146646NB-I00 funded by MICIU/AEI/10.13039/501100011033 and by ESF+.}



\begin{abstract}
We study commutators of the Riesz potential $I_\alpha$ with functions $b$ in the capacitary space $\mathrm{BMO}^\beta(\mathbb{R}^n)$, defined through the Hausdorff content $\hcal^\beta_\infty$. We prove a Chanillo-type theorem characterising $\mathrm{BMO}^\beta(\mathbb{R}^n)$ via the boundedness of the commutator $[b,I_\alpha]$ on capacitary Lebesgue spaces. In addition, we obtain the endpoint estimate in the form of a capacitary modular weak-type inequality. These results follow from a pointwise estimate for the $\beta$-dimensional sharp maximal function of the commutator, together with a capacitary Fefferman-Stein inequality recently proved in \cite{ChenClaros}
\end{abstract}

\maketitle

\section{Introduction and main results}

Commutators of singular integral operators play a central role in harmonic analysis and partial differential equations. The classical theorem of Coifman, Rochberg, and Weiss \cite{MR0412721} asserts that for any Calderón-Zygmund operator $T$ and any locally integrable function $b$, the commutator
\begin{equation*}
	[b,T]f := b\,Tf - T(bf)
\end{equation*}
is bounded on $L^p(\mathbb{R}^n)$ for $1<p<\infty$ if and only if $b\in \operatorname{BMO} (\mathbb{R}^n)$. This result not only characterises the space of functions of bounded mean oscillation $(\operatorname{BMO})$ but also reveals its intrinsic connection with the structure of singular integrals.

A natural extension of this problem concerns commutators of fractional integral operators. For $0<\alpha<n$, the Riesz potential
\begin{equation*}
	I_\alpha f(x) = \int_{\mathbb{R}^n} \frac{f(y)}{|x-y|^{n-\alpha}}\,dy
\end{equation*}
maps $L^p(\mathbb{R}^n)$ to $L^q(\mathbb{R}^n)$, where $1<p<\frac{n}{\alpha}$ and $\tfrac{1}{q} = \tfrac{1}{p}-\tfrac{\alpha}{n}$. Chanillo \cite{chanillo} proved that the corresponding Riesz commutator
\begin{equation*}
	[b,I_\alpha]f = b\,I_\alpha f - I_\alpha(bf)
\end{equation*}
is bounded from $L^p(\mathbb{R}^n)$ to $L^q(\mathbb{R}^n)$ precisely when $b\in\operatorname{BMO} (\mathbb{R}^n)$, and moreover
\begin{equation*}
\| [b,I_\alpha] : L^p(\mathbb{R}^n)\longrightarrow L^q(\mathbb{R}^n) \|
\simeq \|b\|_{\operatorname{BMO} (\mathbb{R}^n)}.
\end{equation*}

Later contributions by Cruz-Uribe and Fiorenza \cite{Cruz-Uribe-Fiorenza} (see also \cite{Ding}) and by Adams and Xiao \cite{MR2922610} refined these results, establishing the endpoint cases $p=1$ and $p=\tfrac{n}{\alpha}$, respectively. In particular, Cruz-Uribe and Fiorenza obtained a modular weak-type estimate of the form
\begin{equation*}
|\{x\in\R^n : |[b,I_\alpha]f(x)|>t\}| 
\le C\, \Psi \left(\int_{\R^n} B\left( \|b\|_{\operatorname{BMO}}\frac{|f(x)|}{t}\right)dx \right),
\end{equation*}
for all $t>0$, where $B(t)=t\log(e+t)$ and $\Psi(t) = [t\log (e+t^{\alpha/n})]^{n /(n-\alpha)}$.

Given $0<\beta \le n$, we define the Hausdorff content $\mathcal{H}^\beta_\infty$ by 
\begin{equation*}
	\mathcal{H}^\beta_\infty(E) := \inf \left\lbrace \sum_{i=1}^\infty \omega_{\beta} \,r_i^\beta : E \subset \bigcup_{i=1}^\infty B(x_i, r_i)\right\rbrace ,
\end{equation*}
where $\omega_{\beta} = \pi^{\beta/2}/\Gamma(\tfrac{\beta}{2}+1\bigr)$ is a normalization constant. The motivation to define function spaces measured with respect to the Hausdorff content $\mathcal{H}^\beta_\infty$ arises naturally from the study of sharp forms of the Sobolev embedding in the critical exponent (see, e.g., \cite{Yudovich, Adams1973, Cianchi:2008, FontanaMorpurgo, MS}). In this borderline regime, one often encounters not only the classical exponential integrability phenomenon, but also refined inequalities that capture concentration effects along lower-dimensional sets. Within this framework, it is natural to revisit function spaces that describe oscillatory behavior through the lens of Hausdorff content. For $0<\beta\le n$, the space of functions of bounded $\beta$-dimensional mean oscillation, denoted by $\operatorname{BMO}^\beta(\mathbb{R}^n)$, introduced in \cite{Chen-Spector},consists of all locally integrable functions $b$ (with respect to the Hausdorff content $\mathcal{H}^\beta_\infty$) such that
\begin{equation*}
	\|b\|_{\operatorname{BMO}^\beta(\mathbb{R}^n)} := \sup_{Q} \inf_{c\in \R}\frac{1}{\ell(Q)^{\beta}} \int_Q |b - c|\, d\mathcal{H}^\beta_\infty < \infty,
\end{equation*}
where the integral with respect to $\mathcal{H}^\beta_\infty$ is understood in the Choquet sense (see Section 2). When $\beta=n$, this definition reduces to the classical space $\operatorname{BMO}(\mathbb{R}^n)$ of John and Nirenberg \cite{JN}; for smaller values of $\beta$, the space $\operatorname{BMO}^\beta(\mathbb{R}^n)$ is smaller than $\operatorname{BMO}(\mathbb{R}^n)$ (see \cite[Corollary 1.6]{Chen-Spector}).

Building on this idea, recent works have developed analytic tools adapted to $\mathcal{H}^\beta_\infty$, including maximal inequalities \cite{chen2023capacitary, ChenClaros}, Poincaré-Sobolev inequalities \cite{Petteri_2023, harjulehto2023sobolev, HuangCaoYangZhuoChoquetPS}, and Riesz potential theory \cite{harjulehto2024hausdorff}. Motivated by these developments, the aim of this paper is to extend the classical Riesz commutator theory of Chanillo \cite{chanillo} and the modular endpoint results of Cruz-Uribe and Fiorenza \cite{Cruz-Uribe-Fiorenza} to this capacitary setting. Specifically, we establish strong and weak-type inequalities for the Riesz commutator $[b, I_\alpha]$ with symbols $b\in {\rm BMO}^\beta(\mathbb{R}^n)$, with respect to the Hausdorff content $\mathcal{H}^\beta_\infty$. Our results recover the Euclidean theory when $\beta=n$.

The first main result of this paper establishes a pointwise estimate for the $\beta$-dimensional sharp maximal function of the commutator. This result can be viewed as an extension of Theorem~1.3 in \cite{Cruz-Uribe-Fiorenza} to the setting of the Choquet integral with respect to Hausdorff content.

\begin{theorem}\label{thm: key pointwise}
Let $B(t)=t\log(e+t)$. Fix $0<\beta\le n$ and $0<\alpha<\beta$. Let $b$ be a measurable function in $\operatorname{BMO}^\beta (\R^n)$ and let $f$ be a nonnegative measurable function. Then there exists a constant $C=C(\alpha,\beta,n)>0$ such that, for every $s>1$ and every $x\in\R^n$,
	\begin{equation}\label{key pointwise estimate}
		\mathcal{M}^\#_\beta ( [b, I_\alpha] f )(x) \le C \| b\|_{\operatorname{BMO}^\beta (\R^n) } \left( \mathcal{M}_{\mathcal{H}_\infty^\beta} ((I_\alpha f)^s) (x)^\frac{1}{s} + \mathcal{M}_{\alpha, B, \mathcal{H}_\infty^\beta} f(x) \right),
	\end{equation}
	where $\mathcal{M}_{\alpha,B,\mathcal{H}^\beta_\infty}$ denotes the fractional Orlicz maximal operator associated with $B$ with respect to the Hausdorff content $\mathcal{H}^\beta_\infty$ (see Section~3). 
    
    Moreover, if $\beta\in(n-\alpha,n]$, then there exists  $C=C(\alpha,\beta,n)>0$ such that, for every $x\in\R^n$,
    \begin{equation}\label{key pointwise estimate 2}
		\mathcal{M}^\#_\beta ( [b, I_\alpha] f )(x) \le C \| b\|_{\operatorname{BMO}^\beta (\R^n) } \left( I_\alpha f (x) + \mathcal{M}_{\alpha, B, \mathcal{H}_\infty^\beta} f(x) \right).
	\end{equation}
\end{theorem}

Here, $\mathcal{M}_{\mathcal{H}^\beta_\infty}$ denotes the $\beta$-dimensional maximal operator introduced in \cite{chen2023capacitary}, defined by
\begin{align*}
 \mathcal{M}_{\mathcal{H}^\beta_\infty} f(x)   = \sup_{x\in Q}  \frac{1}{\mathcal{H}^{\beta}_{\infty} (Q)} \int_{Q} |f| \,\,d\mathcal{H}^{\beta}_{\infty},
\end{align*}
where the supremum is taken over all cubes in $\R^n$ with sides parallel to the coordinate axes and with $x\in Q$. Similarly,  $\mathcal{M}^{\#} _\beta$ denotes the $\beta$-dimensional sharp maximal function studied in \cite{ChenClaros}, defined as
\begin{align*}
\mathcal{M}^{\#} _\beta f(x) := \sup_{x\in Q} \inf_{c \in \mathbb{R}}  \frac{1}{\ell(Q)^\beta} \int_Q |f-c| \, d \mathcal{H}^{\beta}_\infty.
\end{align*}

\begin{remark}\label{Rmk pointwise}
By adapting the proof of Theorem~\ref{thm: key pointwise}, one can avoid the use of Orlicz maximal functions at the expense of obtaining a weaker inequality (see Remark \ref{Rmk pointwise proof}). Specifically, for every $s,t>1$, there exists a constant $C=C(\alpha,\beta,n,s,t)>0$ such that for all $x\in\R^n$,
	\begin{equation*}
		\mathcal{M}^\#_\beta ( [b, I_\alpha] f )(x) \le C \| b\|_{\operatorname{BMO}^\beta (\R^n) } \left( \mathcal{M}_{\mathcal{H}_\infty^\beta} ((I_\alpha f)^s) (x)^\frac{1}{s} + \mathcal{M}_{\alpha t, \mathcal{H}_\infty^\beta} (f^t)(x)^\frac{1}{t} \right),
	\end{equation*}
	where 
	\begin{align*}
 \mathcal{M}_{\alpha , \mathcal{H}^\beta_\infty} f(x)   = \sup_{x\in Q}  \frac{\ell(Q)^\alpha}{\mathcal{H}^{\beta}_{\infty} (Q)} \int_{Q} |f| \,\,d\mathcal{H}^{\beta}_{\infty}.
\end{align*}
\end{remark}

By combining the pointwise inequality stated above with the capacitary Fefferman-Stein inequality from \cite{ChenClaros} (see Theorem \ref{thm: beta FS} below), and using the boundedness of both the fractional integral $I_\alpha$ and the fractional maximal operator $M_{\alpha,\mathcal{H}_\infty^\beta}$ on capacitary Lebesgue spaces obtained in \cite{harjulehto2024hausdorff}, we obtain the following result, which extends the characterization theorem of Chanillo \cite{chanillo} to the present context of Hausdorff content.

\begin{theorem}\label{thm: characterization BMO commutator}
Let $0<\beta\le n$ and $0<\alpha<\beta$. For $1<p<\frac{\beta}{\alpha}$, define the exponent $q$ by the relation
\begin{equation*}
	\frac{1}{p}-\frac{1}{q}= \frac{\alpha}{\beta}.
\end{equation*}
Let $b$ be a measurable function such that the commutator $[b, I_\alpha]$ is well defined. Then, the following statements are equivalent:
\begin{enumerate}[label=(\roman*)]
	\item $b\in \operatorname{BMO}^\beta(\R^n)$.
	\item There exists a constant $C=C(\alpha, \beta, n, p)>0$ such that
		\begin{equation}\label{eq: char BMO commutator}
			\left( \int_{\R^n} \left| [b, I_\alpha] f \right|^q d\mathcal{H}_\infty^\beta \right)^{\frac{1}{q}} \le C  \left( \int_{\R^n} \left|f \right|^p d\mathcal{H}_\infty^\beta \right)^{\frac{1}{p}},
		\end{equation}
	for all bounded functions $f$ with compact support.
\end{enumerate}

Furthermore, the norms satisfy the equivalence
\begin{equation*}
	\| [b,I_\alpha]\|_{(p,q)} \simeq_{n,\alpha, \beta, p } \|b\|_{\operatorname{BMO}^\beta(\R^n)},
\end{equation*}
where $\| [b,I_\alpha]\|_{(p,q)} $ is the best possible constant $C$ in \eqref{eq: char BMO commutator}. 
\end{theorem}

\begin{remark}\label{rem: extension}
Let $1\le p<\infty$ and consider the capacitary Lebesgue space
\begin{equation*}
L^{p}\bigl(\mathcal{H}^{\beta}_{\infty} \bigr)
:=
\left\{
f:\ f \text{ is }\mathcal{H}^{\beta}_{\infty}\text{-quasicontinuous and}\ 
\|f\|_{L^{p}(\mathcal{H}^{\beta}_{\infty})}
<\infty
\right\},
\end{equation*}
where we refer to \cite[p.~2]{Chen-Spector} for the definition of
$\mathcal{H}^{\beta}_{\infty}$-quasicontinuity. This space is endowed with the norm 
\begin{equation*}
    \|f\|_{L^{p}(\mathcal{H}^{\beta}_{\infty})}
:=
\left(\int_{\R^n}|f|^{p}\,d\mathcal{H}^{\beta}_{\infty}\right)^{1/p}.
\end{equation*}

In the implication $(i)\Rightarrow (ii)$ in Theorem~\ref{thm: characterization BMO commutator}, the commutator estimate is proved for the dense subspace $L^{p}\bigl(\mathcal{H}^{\beta}_{\infty}\bigr)\cap C_c(\R^n)$; more precisely, there exists $C>0$ such that
\begin{equation}\label{eq:boundedness_dense_subspace}
    \left( \int_{\R^n} \left| [b, I_\alpha] f \right|^q d\mathcal{H}_\infty^\beta \right)^{\frac{1}{q}} \le C  \|b\|_{\operatorname{BMO}^\beta(\R^n)}\left( \int_{\R^n} \left|f \right|^p d\mathcal{H}_\infty^\beta \right)^{\frac{1}{p}},
\end{equation}
for every $f\in L^{p}\bigl(\mathcal{H}^{\beta}_{\infty}\bigr)\cap C_c(\R^n)$. Since this subspace is dense in $L^{p}\bigl(\mathcal{H}^{\beta}_{\infty}\bigr)$ (see \cite{AdamsChoquet1} and \cite[p.~6]{basak2025uncenteredfractionalmaximalfunctions}), the inequality \eqref{eq:boundedness_dense_subspace} extends to all
$f\in L^{p}\bigl(\mathcal{H}^{\beta}_{\infty}\bigr)$.
\end{remark}

In the Euclidean setting, Cruz-Uribe and Fiorenza \cite{Cruz-Uribe-Fiorenza} established a modular weak-type inequality for the Riesz commutator at the endpoint case. We extend their result to the capacitary framework by replacing the Lebesgue measure with the Hausdorff content and considering symbols in $\operatorname{BMO}^\beta(\mathbb{R}^n)$. A key ingredient in our proof is the boundedness of the capacitary fractional maximal operator $\mathcal{M}_{\alpha, B, \mathcal{H}^\beta_\infty}$, which will be studied in Section 3.

\begin{theorem}\label{thm: modular weak}
Let $0<\beta \le n$, $0<\alpha<\beta$, and assume that $\beta\in (n-\alpha , n]$. Let $b$ be a measurable function in $\operatorname{BMO}^\beta (\R^n)$. Define $B(t)= t \log (e+t)$ and $\Psi(t) = [t\log (e+t^{\alpha/\beta})]^{\beta /(\beta-\alpha)}$. Then, there exists a constant $C=C(\alpha, \beta, n)>0$ such that for every bounded measurable function $f$ with compact support, we have
\begin{align}\label{eq: modular weak}
	\mathcal{H}_\infty^\beta \big( \{ x\in \R^n : |[b, I_\alpha ]f(x)|>t \} \big) \le C\, \Psi \left( \int_{\R^n} B \left( \|b\|_{\operatorname{BMO}^\beta(\R^n)}  \frac{|f(x)|}{t} \right) d\mathcal{H}_\infty^\beta \right)
\end{align}
for all $t>0$.
\end{theorem}

\begin{remark}
	The restriction $\beta\in(n-\alpha,n]$ is intrinsic to the interaction between the Riesz potential $I_\alpha$ and the Hausdorff content $\mathcal{H}^\beta_\infty$. Indeed, if $\beta\le n-\alpha$, the kernel $|x|^{\alpha-n}$ is not integrable near the origin in the Choquet sense with respect to $\mathcal{H}^\beta_\infty$, and consequently $I_\alpha f$ may be infinite on sets of dimension $n-\alpha$ even for reasonable data; see \cite[p.~772]{Adams1975}. This non-integrability phenomenon in the capacitary setting was already observed in \cite{ChenClaros}; we refer the reader to \cite[Remark 1.9]{ChenClaros} for a more detailed discussion and further references.
\end{remark}

\begin{remark}
When $\beta=n$, Theorem~\ref{thm: modular weak} reduces to the endpoint modular estimate of Cruz-Uribe and Fiorenza, who proved that the corresponding inequality is sharp: if it holds with
$\Psi$ replaced by an increasing function $\Psi_0$, then necessarily $\Psi(t/\gamma)\le K\,\Psi_0(t)$ for all $t>0$ and suitable constants $\gamma,K>0$. In particular, the logarithmic growth encoded in $\Psi$ cannot be improved even in the Lebesgue case.
\end{remark}

\begin{remark}
Although throughout the paper we focus on the first-order commutator $[b,I_\alpha]$, the argument leading to Theorem~\ref{thm: key pointwise} can be modified to extend to the iterated commutators $(I_\alpha)^m_b$, $m\in\mathbb{N}$ (see Remark~\ref{rem:iteratedcommutators} for the precise definition), yielding the corresponding variants of our strong-type and endpoint estimates. We do not pursue these extensions; see Remark~\ref{rem:iteratedcommutators}.
\end{remark}

\subsection*{Outline of the paper}
The paper is organized as follows. In Section~2, we collect the necessary background on the Hausdorff content $\mathcal{H}^\beta_\infty$, Choquet integration, and the capacitary Fefferman-Stein inequality for the sharp maximal function. In Section~3, we introduce the fractional Orlicz maximal operator associated with $\mathcal{H}^\beta_\infty$ and establish the modular weak-type bounds that will be essential in the endpoint arguments. Section~4 is devoted to recalling the $L^p-L^q$ mapping properties of the Riesz potential $I_\alpha$ and related fractional maximal operators in the Hausdorff content setting, along with some auxiliary estimates. In Section~5, we prove Theorem \ref{thm: key pointwise}, which constitutes the main technical ingredient of the paper. In Section~6, we prove the characterization of $\operatorname{BMO}^\beta$ given in Theorem \ref{thm: characterization BMO commutator}. Finally, in Section~7, we prove the endpoint modular weak-type estimate given in Theorem \ref{thm: modular weak}.

\subsection*{Notation}
As usual, $C$ denotes a positive constant, possibly varying from line to line. We write $C_{\alpha, \beta, ...}$ to denote a constant depending only on $\alpha, \beta,...$. 

\section{Preliminaries and known results}

\subsection{Hausdorff content and Choquet integrals}

 In this section, we recall some basic facts concerning the Choquet integral with respect to Hausdorff content and dyadic Hausdorff content. Most of these results can be found in \cite{Chen-Spector}. For a comprehensive and well-written introduction to the subject, we also refer the reader to \cite{PS_2023}.

The Choquet integral of a non-negative function \( f \) over a set \( \Omega \subseteq \mathbb{R}^n\) with respect to an outer measure \( H \) is defined by
\begin{align*}
	\int_\Omega f \, dH := \int_0^\infty H(\{x \in \Omega : f(x) > t\}) \, dt.
\end{align*}
Let $0<\beta\le n$ and let $Q\subseteq\mathbb{R}^n$ be a cube.
We write $\mathcal{H}^{\beta,Q}_{\infty}$ for the dyadic $\beta$-dimensional Hausdorff
content relative to $Q$. More precisely, the dyadic Hausdorff content $\mathcal{H}^{\beta,Q}_{\infty}(E)$
of a set $E\subseteq\mathbb{R}^n$ is defined by
\[
\mathcal{H}^{\beta,Q}_{\infty}(E)
:=
\inf\left\{
\sum_{j} \ell(Q_j)^{\beta}
:\;
E\subseteq \bigcup_j Q_j,\;
Q_j\in\mathcal{D}(Q)
\right\},
\]
where $\mathcal{D}(Q)$ denotes the collection of all dyadic cubes generated by $Q$
and $\ell(Q_j)$ is the sidelength of $Q_j$. While the Hausdorff content $\mathcal{H}^{\beta}_{\infty}$ is not strongly subadditive if $\beta <n$, it can be shown that the dyadic Hausdorff content $\mathcal{H}^{\beta, Q}_\infty$ is strongly subadditive for any cube $Q \subseteq \mathbb{R}^d$ and thus satisfies
\begin{align*}
     \int_{\mathbb{R}^n} \sum^\infty_{j=1} f_j \;d \mathcal{H}^{\beta, Q}_\infty \leq \sum^\infty_{j=1}\int_{\mathbb{R}^n} f_j \;d \mathcal{H}^{\beta,Q}_\infty
\end{align*}
(see \cite[Proposition 3.5 and 3.6]{STW} for usual dyadic case and \cite[Proposition 2.6 and Proposition 2.10]{Chen-Spector} for general cube $Q$). Moreover, there exists a constant $C_\beta>0$, depending only on $\beta$, such that for every
cube $Q\subseteq\mathbb{R}^n$,
\begin{align}\label{equivalenceofdyadic}
	\frac{1}{C_\beta}\,\mathcal{H}^{\beta,Q}_{\infty}(E)
	\le
	\mathcal{H}^{\beta}_{\infty}(E)
	\le
	C_\beta\,\mathcal{H}^{\beta,Q}_{\infty}(E),
\end{align}
for all $E\subseteq\mathbb{R}^n$(see \cite[Proposition~2.3]{YangYuan} and
\cite[Proposition~2.11]{Chen-Spector}). One advantage of passing between the Choquet integral with respect to the Hausdorff
content and its dyadic counterpart is that, for any $c\ge0$ and any nonnegative function
$f$, one has
\begin{align}\label{averageestimate}
     \bigl| f_{Q,\beta} - c \bigr|
    \le
    \frac{1}{\mathcal{H}^{\beta,Q_0}_{\infty}(Q)}
    \int_Q |f - c| \, d\mathcal{H}^{\beta,Q_0}_{\infty},
\end{align}
where $f_{Q,\beta}$ is defined as
\begin{align*}
      f_{Q,\beta}
  :=
  \frac{1}{\mathcal{H}^{\beta,Q_0}_{\infty}(Q)}
  \int_{Q} f \, d \mathcal{H}^{\beta,Q_0}_{\infty}
\end{align*}
(see \cite[Lemma~2.3]{basak2025uncenteredfractionalmaximalfunctions}).
The next lemma records basic properties of the Choquet integral associated with Hausdorff content. It appears in \cite[p.~5]{Petteri_2023}; proofs may be found in \cite{AdamsChoquet1} and \cite[Chapter~4]{AdamsMorreySpacebook}.

\begin{lemma}\label{basicChoquet}
Let $0<\beta\leq n$ with $n \in \mathbb{N}$, and let $\Omega \subseteq \mathbb{R}^n$. Then the following statements hold:
\begin{enumerate}[label=(\roman*)]
	\item For \( a \geq 0 \) and non-negative functions \( f\) on $\mathbb{R}^n$, we have
		\begin{align*}
		\int_\Omega a \, f(x)\, d\mathcal{H}^\beta_\infty = a \int_\Omega f(x)\, d\mathcal{H}^\beta_\infty;
		\end{align*}
	\item For non-negative functions \( f_1 \) and \( f_2 \) on $\mathbb{R}^n$, we have
	\begin{align*}
 	 \int_\Omega f_1(x) + f_2(x) \, d\mathcal{H}^\beta_\infty \leq 2 \left( \int_\Omega f_1(x) \, d\mathcal{H}^\beta_\infty + \int_\Omega f_2(x) \, d\mathcal{H}^\beta_\infty \right);
	\end{align*}
	\item Let \( 1 < p<\infty , \) and let $p'$ be defined by \( \frac{1}{p} + \frac{1}{p'} = 1 \). Then for non-negative functions \( f_1 \) and \( f_2 \) on $\mathbb{R}^n$, we have
	\begin{align*}
 	\int_\Omega f_1(x) f_2(x) \, d\mathcal{H}^\beta_\infty \leq 2 \left( \int_\Omega f_1(x)^p \, d\mathcal{H}^\beta_\infty \right)^{\frac{1}{p}} \left( \int_\Omega f_2(x)^{p'} \, d\mathcal{H}^\beta_\infty \right)^{\frac{1}{p'}}.
	\end{align*}
\end{enumerate}
\end{lemma}

We next state a lemma from \cite[Lemma 2.6]{ChenClaros}, which is a minor modification of a result in \cite[Lemma 2]{OV}.

\begin{lemma}\label{packing2}
Suppose that $\{Q_j\}$ is a family of non-overlapping dyadic cubes subordinate to some dyadic lattice $\mathcal{D}(Q_0)$. Then there exists a subfamily $\{Q_{j_k}\}$ and a family of non-overlapping ancestors $\tilde{Q}_k$ such that
\begin{enumerate}
	\item 
	\begin{align*}
		\bigcup_{j} Q_j \subset \bigcup_{k} Q_{j_k} \cup \bigcup_{k} \tilde{Q}_k
	\end{align*}
	\item
	\begin{align*}
		\sum_{Q_{j_k} \subset Q}\ell(Q_{j_k})^\beta \leq 2\ell (Q)^\beta, \text{ for each dyadic cube } Q.
	\end{align*}
	\item
	\begin{align*}
  	\mathcal{H}^{\beta, Q_0}_\infty(\cup_j Q_j) & \leq  \sum_{k, Q_{j_k} \not \subseteq \tilde{Q}_m} \ell(Q_{j_k})^\beta + \sum_k \ell(\tilde{Q}_k)^\beta \\ 
    &\leq \sum_k \ell(Q_{j_k})^\beta \\
    & \leq 2 \mathcal{H}^{\beta, Q_0}_\infty (\cup_j Q_j).
 	\end{align*}
\end{enumerate}
\end{lemma}
\begin{remark}
Condition~(2) in Lemma~\ref{packing2} is usually referred to as a \emph{packing condition}.
In particular, let $\{Q_k\}_k \subset \mathcal{D}(Q_0)$ be a collection of dyadic cubes such that
\[
\sum_{Q_k \subset Q} \ell(Q_k)^\beta \le 2\,\ell(Q)^\beta
\qquad
\text{for every dyadic cube } Q.
\]
Then, for every $f\ge 0$, one has
\begin{align}\label{packingestimate}
\sum_k \int_{Q_k} f \, d\mathcal{H}^{\beta,Q_0}_\infty
\le
2 \int_{\bigcup_k Q_k} f \, d\mathcal{H}^{\beta,Q_0}_\infty .
\end{align}
\end{remark}

The next lemma is taken from \cite[Lemma 2.2]{chen2023capacitary} and \cite[Proposition 2.3]{harjulehto2023sobolev}, extending \cite[Lemma 3]{OV}.

\begin{lemma}\label{alphabetachange}
Let $f \geq 0$ and $0 < \alpha \leq \beta \leq n \in \mathbb{N}$. Then
\begin{align*}
    \int_{\mathbb{R}^n} f \, d\mathcal{H}^\beta_\infty \leq \frac{\beta }{\alpha} \left( \int_{\mathbb{R}^n} f^{\frac{\alpha}{\beta}} \, d\mathcal{H}^\alpha_\infty \right)^{\frac{\beta}{\alpha}}.
\end{align*}
\end{lemma}

We will use the following $\beta$-dimensional Fefferman-Stein inequality established in \cite[Theorem 1.1]{ChenClaros}.

\begin{theorem}\label{thm: beta FS}
    Let $0 <  \beta  \leq n \in \mathbb{N}$, $0<p_0<\infty$ and $f$ locally integrable with respect to $\mathcal{H}^{\beta}_\infty$  such that \begin{equation}\label{newassumptionThm1.2}
    \sup_{0< t \le N} t^{p_0} \mathcal{H}^{\beta}_\infty ( \{ x\in \R^n : \mathcal{M}_{\mathcal{H}^\beta_\infty} f (x)>t\}) <\infty \qquad \text{ for all } N>0.
\end{equation}
    Then there exists a constant $C = C(p,n,\beta)>0$ such that 
   \begin{align*}
 \left(  \int_{\mathbb{R}^n} \left( \mathcal{M}_{\mathcal{H}^\beta_\infty} f  \right)^p \; d\mathcal{H}^\beta_\infty \right)^{\frac{1}{p}} \le C \left( \int_{\mathbb{R}^n} \left( \mathcal{M}^{\#} _\beta f \right)^p \; d\mathcal{H}^\beta_\infty \right)^{\frac{1}{p}}
   \end{align*}
   for all $p_0<p<\infty$.
\end{theorem}

Let \( 0 < \beta \le n \) and let \( Q_0 \subseteq \mathbb{R}^n \) be a cube.
The dyadic maximal operator associated with the dyadic Hausdorff content
\( \mathcal{H}^{\beta,Q_0}_\infty \), introduced in \cite{Chen-Spector},
is defined by
\[
 \mathcal{M}_{\mathcal{H}^{\beta,Q_0}_\infty} f (x)
 :=
 \sup_{x\in Q }
 \frac{1}{\ell(Q)^\beta}
 \int_Q |f| \, d\mathcal{H}^{\beta,Q_0}_\infty,
\]
where the supremum is taken over all dyadic cubes
\( Q \in \mathcal{D}(Q_0) \) containing \( x \).
The corresponding weak-type estimate was also established in
\cite{Chen-Spector}. To prove the modular weak point inequality stated in Theorem \ref{thm: modular weak}, we require the following modification of the previous result.

\begin{lemma}
	Let $\varphi : (0,\infty) \longrightarrow (0,\infty)$ be a doubling function, that is, there exists $C_\varphi>0$ such that $\varphi(2t)\le C_\varphi \varphi(t)$ for every $t>0$. Then, there exists a constant $C=C(\beta, \varphi)>0$ such that 
	\begin{equation}\label{weak orlicz FS}
		\sup_{\lambda>0} \varphi(\lambda) \mathcal{H}_\infty^{\beta} (\{ x :  \mathcal{M}_{\mathcal{H}_\infty^{\beta, Q_0}} f(x)>\lambda \}) \le C \sup_{\lambda>0} \varphi(\lambda) \mathcal{H}_\infty^{\beta} (\{ x  :  \mathcal{M}_{\beta}^\# f(x)>\lambda \})	\end{equation}
	for all functions $f$ such that the left-hand side is finite. 
\end{lemma}

The proof adapts the arguments of \cite[Lemma 2.1]{Perez}, relying on the good-lambda inequality established in \cite{ChenClaros}. We include the details for the reader's convenience. 

\begin{proof}
	In view of the equivalence \eqref{equivalenceofdyadic}, it suffices to establish \eqref{weak orlicz FS} for the dyadic Hausdorff content $\mathcal{H}_\infty^{\beta, Q_0}$. By the good-lambda inequality proved in \cite[Theorem 3.1]{ChenClaros}, there exists a constant $C_\beta>0$ such that 
	\begin{align*}
		\mathcal{H}_\infty^{\beta, Q_0} (\{ x\in \R^n :  \mathcal{M}_{\mathcal{H}_\infty^{\beta, Q_0}} f(x)>\lambda \}) \le & \mathcal{H}_\infty^{\beta,Q_0} (\{ x\in \R^n :  \mathcal{M}_{\beta}^\# f(x)> \varepsilon \lambda \})\\
		& + \varepsilon  \, C_\beta  \mathcal{H}_\infty^{\beta, Q_0} (\{ x\in \R^n :  \mathcal{M}_{\mathcal{H}_\infty^{\beta, Q_0}} f(x)>c_\beta \lambda \})
	\end{align*}
	for all $\lambda>0$ and $\varepsilon>0$, where $c_\beta = \frac{1}{2^{\beta+2}}$. Using the doubling property of $\varphi$, there exists a constant $C_1=C_1(\beta, \varphi)>0$ such that $\varphi(\lambda)= \varphi(2^{\beta+2} c_\beta \lambda) \le C_1 \varphi(c_\beta \lambda) $ for all $\lambda>0$. We choose $\varepsilon>0$ sufficiently small so that $\varepsilon C_\beta C_1 = \frac{1}{2}$. Invoking the doubling property again, there exists a constant $C_2= C_2(\beta, \varphi)>0$ such that $\varphi(\lambda)= \varphi(\frac{1}{\varepsilon} \varepsilon \lambda) \le C_2 \varphi(\varepsilon \lambda) $. Consequently, we have
	\begin{align*}
		 \varphi(\lambda) \mathcal{H}_\infty^{\beta, Q_0} (\{ x :  \mathcal{M}_{\mathcal{H}_\infty^{\beta, Q_0}} f(x)>\lambda \}) \le & C_2\varphi(\varepsilon \lambda) \mathcal{H}_\infty^{\beta, Q_0} (\{ x :  \mathcal{M}_{\beta}^\# f(x)> \varepsilon \lambda \})\\
		& + \frac{1}{2} \varphi(c_\beta \lambda) \mathcal{H}_\infty^{\beta, Q_0} (\{ x :  \mathcal{M}_{\mathcal{H}_\infty^{\beta, Q_0}} f(x)>c_\beta \lambda \})\\
		\le & C_2 \sup_{\lambda>0} \varphi(\varepsilon \lambda) \mathcal{H}_\infty^{\beta, Q_0} (\{ x :  \mathcal{M}_{\beta}^\# f(x)> \varepsilon \lambda \})\\
		& + \frac{1}{2} \sup_{\lambda>0} \varphi(c_\beta \lambda) \mathcal{H}_\infty^{\beta, Q_0} (\{ x :  \mathcal{M}_{\mathcal{H}_\infty^{\beta, Q_0}} f(x)>c_\beta \lambda \})\\
		= & C_2 \sup_{\lambda>0} \varphi( \lambda) \mathcal{H}_\infty^{\beta, Q_0} (\{ x :  \mathcal{M}_{\beta}^\# f(x)> \lambda \})\\
		& + \frac{1}{2} \sup_{\lambda>0} \varphi( \lambda) \mathcal{H}_\infty^{\beta, Q_0} (\{ x :  \mathcal{M}_{\mathcal{H}_\infty^{\beta, Q_0}} f(x)> \lambda \}).
	\end{align*}
	Taking the supremum over $\lambda>0$ on the left-hand side and assuming it is finite, we may absorb the second term on the right-hand side to conclude the proof of \eqref{weak orlicz FS} with $C=2C_2$.
\end{proof}

\subsection{On $\beta$-dimensional bounded mean oscillation spaces}
For a function $b\in \operatorname{BMO}^\beta(\R^n)$ and any cube $Q$, we define the quantity
\begin{equation*}
	b_Q = \argmin_{c\in \R} \frac{1}{\ell(Q)^\beta} \int_Q |b-c| d\mathcal{H}_\infty^\beta.
\end{equation*}

We observe that, for every $b\in \operatorname{BMO}^\beta(\mathbb{R}^n)$ and every cube $Q$, the infimum in the definition of $\|b\|_{ \operatorname{BMO}^\beta(\R^n) }$ is always attained; that is, the previous $\argmin$ exists. See \cite[p. 983]{Chen-Spector} for a detailed discussion.

\begin{lemma}\label{bmo cQ}
	Let $0<\beta\le n \in \mathbb{N}$ and $b\in \operatorname{BMO}^\beta(\R^n)$. Then, there exist $C_\beta>0$ such that
	\begin{equation*}
		|b_{2^kQ} - b_Q| \le C_\beta k \|b\|_{ \operatorname{BMO}^\beta(\R^n) }
	\end{equation*}
	for each cube $Q$ and each $k\in \mathbb{N}$. 
\end{lemma}

\begin{proof}
	Firstly, we observe that
	\begin{align*}
		|b_{2Q}-b_Q| = & \frac{C_\beta }{\ell(Q)^\beta} \int_Q |b_{2Q}-b_Q| d\mathcal{H}_\infty^\beta\\
		\le & \frac{C_\beta}{\ell(Q)^\beta} \int_Q |b-b_{2Q}| d\mathcal{H}_\infty^\beta + \frac{C_\beta}{\ell(Q)^\beta} \int_Q |b-b_Q| d\mathcal{H}_\infty^\beta\\
		\le & \frac{C_\beta}{\ell(2Q)^\beta} \int_{2Q} |b-b_{2Q}| d\mathcal{H}_\infty^\beta + \frac{C_\beta}{\ell(Q)^\beta} \int_Q |b-b_Q| d\mathcal{H}_\infty^\beta\\
		\le & C_\beta \|b\|_{ \operatorname{BMO}^\beta(\R^n) }
	\end{align*}
	Consequently,
	\begin{align*}
		|b_{2^kQ} - b_Q| \le & \sum_{i=1}^k |b_{2^iQ} - b_{2^{i-1}Q}| \le  \sum_{i=1}^k C_\beta \|b\|_{ \operatorname{BMO}^\beta(\R^n) }
		=  C_\beta k  \|b\|_{ \operatorname{BMO}^\beta(\R^n) }.
	\end{align*}
\end{proof}

We require the following result, which is a minor modification of \cite[Corollary 1.4]{Chen-Spector}.

\begin{theorem}\label{exp BMO}
	Let $0<\beta\le n \in \mathbb{N}$ and $b\in \operatorname{BMO}^\beta(\R^n)$. There exist constants $C_1,C_2>0$ such that 
	\begin{equation*}
		\frac{1}{\ell(Q)^\beta} \int_Q \exp \left( \frac{|b-b_Q|}{C_1 \| b\|_{\operatorname{BMO}^\beta(\R^n) }}\right) d\mathcal{H}_\infty^\beta \le C_2,
	\end{equation*}
	for each cube $Q$. 
\end{theorem}

We also recall the following result from \cite[Corollary 1.5]{Chen-Spector}. 

\begin{theorem}\label{thm: JN p}
	Let $0<\beta\le n \in \mathbb{N}$ and $b\in \operatorname{BMO}^\beta(\R^n)$. There exists a constant $C=C(\beta)>0$ such that 
	\begin{equation*}
		\sup_Q \inf_{c\in \R} \left( \frac{1}{\ell(Q)^\beta} \int_Q |b-c|^p d\mathcal{H}_\infty^\beta\right)^\frac{1}{p} \le C p \|b\| _{\operatorname{\operatorname{BMO}^\beta} }
	\end{equation*}
	for all $p\ge 1$. 
\end{theorem}

We also require the following basic properties of $\operatorname{BMO}^\beta$ functions.

\begin{lemma}\label{basicbmobeta}
    Let $0<\beta \leq n \in \mathbb{N} $. Then there exists a constant $C_\beta> 0$ depending only on $\beta$ such that
\begin{enumerate}[label=(\roman*)]
    \item $\Vert f+ g \Vert_{\operatorname{BMO}^\beta(\mathbb{R}^n)} \leq 2 ( \Vert f \Vert_{\operatorname{BMO}^\beta(\mathbb{R}^n)} + \Vert g \Vert_{\operatorname{BMO}^\beta(\mathbb{R}^n)})$ 
    \item  $\Vert \lambda f \Vert_{\operatorname{BMO}^\beta(\mathbb{R}^n)} = |\lambda| \Vert f \Vert_{\operatorname{BMO}^\beta(\mathbb{R}^n)}$ for any $\lambda \in \mathbb{R} $
    \item $\Vert |f| \Vert_{\operatorname{BMO}^\beta(\mathbb{R}^n)} \leq C_\beta \Vert f \Vert_{\operatorname{BMO}^\beta(\mathbb{R}^n)}$
    \item  $\Vert \max(f,g)\Vert_{\operatorname{BMO}^\beta(\mathbb{R}^n)} \leq C_\beta ( \Vert f \Vert_{\operatorname{BMO}^\beta(\mathbb{R}^n)} + \Vert g \Vert_{\operatorname{BMO}^\beta(\mathbb{R}^n)} ) $
    \item  $\Vert \min(f,g)\Vert_{\operatorname{BMO}^\beta(\mathbb{R}^n)} \leq C_\beta ( \Vert f \Vert_{\operatorname{BMO}^\beta(\mathbb{R}^n)} + \Vert g \Vert_{\operatorname{BMO}^\beta(\mathbb{R}^n)} ) $
\end{enumerate}
\end{lemma}
\begin{proof}
    The proof of $(i)$ and $(ii)$ follows immediately from Lemma \ref{basicChoquet}.

The proof of $(iii)$ is a minor modification of Lemma~2.6 in \cite{basak2025uncenteredfractionalmaximalfunctions}.  We include the details for completeness. We define $\operatorname{BMO}^\beta$ with respect to the dyadic Hausdorff content by
\begin{align*}
\|f\|_{\operatorname{BMO}^\beta(\mathcal{H}^{\beta,Q_0}_{\infty})}:= \sup_{Q\subseteq \mathbb{R}^n} \inf_{c \in \mathbb{R}} \left( \frac{1}{\mathcal{H}^{\beta,Q_0}_{\infty}(Q)} \int_Q |f(y) - c| \; d\mathcal{H}^{\beta,Q_0}_{\infty}(y) \right)<\infty,
\end{align*}
where $Q_0 : = [0,1)^n$. It follows from (\ref{equivalenceofdyadic}) that
\begin{align*}
\|f\|_{\operatorname{BMO}^\beta(\mathcal{H}^{\beta,Q_0}_{\infty})} \cong \Vert f \Vert_{\operatorname{BMO}^\beta}.
\end{align*}
For cubes \(Q' \subseteq \mathbb{R}^n\), we set
\begin{align*}
c_{Q'}= \argmin_{c \in \mathbb{R}} \frac{1}{\mathcal{H}^{\beta,Q_0}_{\infty}(Q')}\int_{Q'} \, |f-c| \;d \mathcal{H}^{\beta,Q_0}_{\infty}.
\end{align*}
Then, for every \(x \in Q'\), we have by (\ref{averageestimate}) that
\begin{align*}
& \left| |f(x)|-\frac{1}{\mathcal{H}^{\beta,Q_0}_{\infty}		(Q')}\int_{Q'}|f(y)| \; d \mathcal{H}^{\beta,Q_0}_{\infty} (y) \right| \\
\leq & \frac{1}{\mathcal{H}^{\beta,Q_0}_{\infty}		(Q')}\int_{Q'} |f(x)-f(y)| \; d \mathcal{H}^{\beta,Q_0}_{\infty}(y) \\
=& \frac{1}{\mathcal{H}^{\beta,Q_0}_{\infty}		(Q')} \int_{Q'} |f(x)-c_{Q'}+c_{Q'}-f(y)| \; d \mathcal{H}^{\beta,Q_0}_{\infty} (y) \\
\leq & |f(x)-c_{Q'}| + \frac{1}{\mathcal{H}^{\beta,Q_0}_{\infty}		(Q')} \int_{Q'} |f(y)-c_{Q'}| \; d \mathcal{H}^{\beta,Q_0}_{\infty}(y)\\
\leq & |f(x)-c_{Q'}| + \|f\|_{\operatorname{BMO}^\beta(\mathcal{H}^{\beta,Q_0}_{\infty})}.
\end{align*}

Integrating the preceding estimate over \(Q'\) with respect to the Hausdorff content \(\mathcal{H}^{\beta}_{\infty}\), and then dividing by \(\mathcal{H}^{\beta}_{\infty}(Q')\), we obtain
\begin{align*}
&\frac{1}{\mathcal{H}^{\beta}_{\infty}		(Q')} \int_{Q'} \left| |f(x)|-\frac{1}{\mathcal{H}^{\beta,Q_0}_{\infty}		(Q')}\int_{Q'}|f| \; d \mathcal{H}^{\beta,Q_0}_{\infty} \right| \; d \mathcal{H}^{\beta}_{\infty} (x) \\ \leq &  \frac{1}{\mathcal{H}^{\beta}_{\infty}		(Q')} \int_{Q'} |f(x)-c_{Q'}| \; d \mathcal{H}^{\beta}_{\infty} (x) + \|f\|_{\operatorname{BMO}^\beta(\mathcal{H}^{\beta,Q_0}_{\infty})}\\
\leq  & \frac{C_\beta}{\mathcal{H}^{\beta,Q_0}_{\infty}		(Q')} \int_{Q'} |f(x)-c_{Q'}| \; d \mathcal{H}^{\beta,Q_0}_{\infty} (x) + \|f\|_{\operatorname{BMO}^\beta(\mathcal{H}^{\beta,Q_0}_{\infty})}
\\
\leq  & (1+ c_\beta) \|f\|_{\operatorname{BMO}^\beta(\mathcal{H}^{\beta,Q_0}_{\infty})}.
\end{align*}
Therefore, for each cube $Q'\subset \mathbb{R}^n$, we obtain
\begin{align*}
\nonumber&\inf_{c\in \mathbb{R}} \frac{1}{\mathcal{H}^{\beta}_{\infty}(Q')} \int_{Q'} \left| |f(x)|-c\right| \; d \mathcal{H}^{\beta}_{\infty} (x) \\ \leq& \frac{1}{\mathcal{H}^{\beta}_{\infty}(Q')} \int_{Q'} \nonumber\left| |f(x)|-\frac{1}{\mathcal{H}^{\beta,Q_0}_{\infty}(Q')}\int_{Q'}|f| \; d \mathcal{H}^{\beta,Q_0}_{\infty} \right| \; d \mathcal{H}^{\beta}_{\infty} (x)\nonumber\\ \leq &
(1+c_\beta) \|f\|_{\operatorname{BMO}^\beta(\mathcal{H}^{\beta,Q_0}_{\infty})}   \nonumber\\
\leq &  C_\beta  \|f\|_{\operatorname{BMO}^\beta}.
\end{align*}
Consequently,
\begin{align*}
    \Vert |f| \Vert_{\operatorname{BMO}^\beta(\mathbb{R}^n)} \leq  C_\beta  \Vert f \Vert_{\operatorname{BMO}^\beta(\mathbb{R}^n)}.
\end{align*}

    The proof of $(iv)$ and $(v)$ follws by $(iii)$ and the equalities
    \begin{align*}
        \max(f,g) = \frac{f+g + |f-g|}{2} \text{ and } \min(f,g) = \frac{f+g - |f-g|}{2}.
    \end{align*}
\end{proof}

An immediate application of Lemma~\ref{basicbmobeta} yields the following result.

\begin{lemma}\label{truncatedBMO}
    Let $0< \beta \leq n$, $b \in \operatorname{BMO}^\beta(\mathbb{R}^n)$ and $k>0 $. Let the truncated function $b_k$ be defined as 
 \begin{align*}
   b_k (x) := \begin{cases}
   k, & \text{if } b(x) > k;\\[6pt]
   b(x), & \text{if } -k \leq b(x) \leq k;\\
   -k & \text{if } b(x) < -k.
   \end{cases}
   \end{align*}
Then there exists a constant $C_\beta>0$, depending only on $\beta$, such that
\begin{align*}
    \Vert b_k \Vert_{\operatorname{BMO}^\beta(\mathbb{R}^n)} \leq C_\beta  \Vert b \Vert_{\operatorname{BMO}^\beta(\mathbb{R}^n)}.
\end{align*}
\end{lemma}

\section{Capacitary Orlicz spaces and maximal functions}

A function $B:[0,\infty)\to[0,\infty)$ is called a Young function if it is continuous, convex, and strictly increasing, satisfying $B(0)=0$ and $\lim_{t\to\infty} B(t)=\infty$. Given two Young functions $A$ and $B$, we write $A(t)\approx B(t)$ if there exist constants $t_0,c_1,c_2>0$ such that $c_1A(t)\le B(t)\le c_2A(t)$ for all $t\ge t_0$. 

A Young function $B$ is said to be doubling if $B(2t)\le C B(t)$ for all $t>0$, and submultiplicative if $B(st)\le C B(s)B(t)$ for all $s,t>0$. Standard examples include $B(t)=t^r$ for $r\ge1$, and more generally $B(t)=t^a[\log(e+t)]^b$ with $a\ge1$ and $b>0$, both of which are submultiplicative.

Given a nonempty set $E\subset\R^n$ and a Young function $B$, we define the Orlicz space $L^B(E,\mathcal{H}_\infty^\beta)$ as the set of all functions $f$ for which $B(|f|/\lambda)$ is integrable on $E$ with respect to the Hausdorff content $\mathcal{H}_\infty^\beta$ for some $\lambda>0$. This space is equipped with the Luxemburg quasinorm
\begin{equation*}
	\|f\|_{L^B(E, \mathcal{H}_\infty^\beta)}=\inf \left\{\lambda>0: \int_E B\left(\frac{|f|}{\lambda}\right) d \mathcal{H}_\infty^\beta \leq 1\right\}.
\end{equation*}
In the specific case where $E$ is a cube $Q$, we also consider the mean Luxemburg quasinorm, defined by
\begin{equation*}
	\|f\|_{B, Q, \mathcal{H}_\infty^\beta}=\inf \left\{\lambda>0: \frac{1}{\ell(Q)^\beta} \int_Q B\left(\frac{|f|}{\lambda}\right) d \mathcal{H}_\infty^\beta \leq 1\right\}.
\end{equation*}

\begin{remark}\label{rem exp BMO}
	With the previous notation, Theorem \ref{exp BMO} may be restated as
	\begin{equation*}
		\| b- b_Q\|_{B, Q, \mathcal{H}_\infty^\beta } \le C \| b\|_{\operatorname{BMO}^\beta(\R^n) },
	\end{equation*}
	for each cube $Q$, where $B(t)=e^t-1$.
\end{remark}

Given a Young function $B$, its complementary Young function is defined as
\begin{equation*}
	\bar{B}(t)=\sup_{s>0}\{s t-B(s)\}, \qquad t>0.
\end{equation*}

\begin{remark}\label{rem: B complementary}
    Let $B(t)=t\log(e+t)$, the complementary Young function $\bar{B}$ satisfies $\bar{B}(t)\lesssim e^t-1$ (see \cite{BennettSharpley}).
\end{remark}

We need the Hölder's inequality for Orlicz spaces with respect to the Hausdorff content.

\begin{lemma}\label{lem: Holder orlicz}
	Given a Young function $B$, then for all functions $f$ and $g$ and all cubes $Q$,

\begin{equation}\label{Holder Orlicz}
	\frac{1}{\ell(Q)^\beta} \int_Q|f g| d \mathcal{H}_\infty^\beta \leq C_\beta \|f\|_{B, Q, \mathcal{H}_\infty^\beta}\|g\|_{\bar{B}, Q, \mathcal{H}_\infty^\beta}.
\end{equation}

\end{lemma}

\begin{proof}
	We have Young's inequality, 
	\begin{equation*}
		xy\le B(x)+\bar{B}(y),
	\end{equation*}
	for all $x,y\ge 0$. Then, we have
	\begin{equation*}
		\frac{|f(x)|}{\|f\|_{B, Q, \mathcal{H}_\infty^\beta}} \frac{|g(x)|}{\|g\|_{\bar{B}, Q, \mathcal{H}_\infty^\beta}} \le B\left(\frac{|f(x)|}{\|f\|_{B, Q, \mathcal{H}_\infty^\beta} }\right) +\bar{B}\left( \frac{|g(x)|}{\|g\|_{\bar{B}, Q, \mathcal{H}_\infty^\beta}}\right).
	\end{equation*}
	Integrating over $Q$,
	\begin{align*}
		\frac{1}{\ell(Q)^\beta} \int_Q \frac{|f|}{\|f\|_{B, Q, \mathcal{H}_\infty^\beta}} \frac{|g|}{\|g\|_{\bar{B}, Q, \mathcal{H}_\infty^\beta}}  d \mathcal{H}_\infty^\beta\le & C_\beta \frac{1}{\ell(Q)^\beta} \int_Q   B\left(\frac{|f(x)|}{\|f\|_{B, Q, \mathcal{H}_\infty^\beta} }\right)  d \mathcal{H}_\infty^\beta\\
		&+ C_\beta \frac{1}{\ell(Q)^\beta} \int_Q   \bar{B}\left( \frac{|g(x)|}{\|g\|_{\bar{B}, Q, \mathcal{H}_\infty^\beta}}\right) d \mathcal{H}_\infty^\beta\\
		\le & 2C_\beta.
	\end{align*}
\end{proof}

These concepts motivate the following definition of the Orlicz fractional maximal operator associated with Hausdorff content.

\begin{definition}
	Let $n \in \mathbb{N}$ and $0<\alpha< \beta \leq n$. The $\beta$-dimensional fractional maximal operator with respect to a Young function $B$ of a function $f$ in $\mathbb{R}^n$ is defined by
	\begin{equation*}
		\mathcal{M}_{\alpha, B, \mathcal{H}_\infty^\beta} f(x) = \sup_{ Q} \chi_Q(x) \ell(Q)^\alpha \|f\|_{B, Q, \mathcal{H}_\infty^\beta}.
	\end{equation*}
	When $\alpha=0$, we simply write $\mathcal{M}_{0, B, \mathcal{H}_\infty^\beta} f = \mathcal{M}_{ B, \mathcal{H}_\infty^\beta} f $. Let $\mathcal{D}$ be the collection of the usual dyadic cubes in $\mathbb{R}^n$, we define the dyadic $\beta$-dimensional fractional maximal operator with respect to a Young function $B$ of a function $f$ as 
	\begin{equation*}
		\mathcal{M}_{\alpha, B, \mathcal{H}_\infty^{\beta, Q_0}}^{\mathcal{D}} f(x) = \sup_{ P\in \mathcal{D}} \chi_P(x) \ell(P)^\alpha \|f\|_{B, P, \mathcal{H}_\infty^{\beta, Q_0}}.
	\end{equation*}

\end{definition}

\begin{remark}\label{remark Holder}
	By setting $g=1$ in the Hölder inequality, we observe that for any Young function $B$ and $0<\alpha<\beta$,
	\begin{equation*}
		\mathcal{M}_{\alpha, \mathcal{H}_\infty^\beta} f(x) \le C_\beta \, \mathcal{M}_{\alpha, B, \mathcal{H}_\infty^\beta} f(x)
	\end{equation*}
	for all $x\in \R^n$.
\end{remark}

\begin{lemma}\label{lemma dyadic orlicz}
    Let $0< \beta \le n$ and $0\le \alpha <\beta$. Let $f$ be a locally integrable function, and let $B$ be a Young function. Suppose that for some cube $Q$ and some parameter $t>0$, we have
	\begin{equation}\label{lemma orlicz eq1}
		\ell(Q)^\alpha \left\| f\right\|_{B, Q, \mathcal{H}_\infty^\beta }>t.
	\end{equation}
	Then, there exists a dyadic cube $P\in \mathcal{D}$ such that $Q\subset 3P$ and a constant $C_{ \beta, n}>0$ such that
	\begin{equation*}
		\ell(P)^\alpha \left\| f\right\|_{B, P, \mathcal{H}_\infty^{\beta, Q_0}}>c_{\beta ,n}t.
	\end{equation*}
\end{lemma}

\begin{proof}
	Let $k\in \Z$ be the unique integer such that $2^{k-1}<\ell(Q)\le 2^k$. There exist at least one and at most $2^n$ dyadic cubes $\{P_j\}_{j=1}^m$ ($1\le m\le 2^n$) of side length $2^k$ that intersect the interior of $Q$. Note that $Q\subset 3P_j$ for each $j$. We claim that one of these cubes, say $P_1$, satisfies
	\begin{equation*}
		\ell(P_1)^\alpha \left\| \chi_{P_1} f\right\|_{B, Q, \mathcal{H}_\infty^{\beta, Q_0}}>\frac{t}{c_\beta 2^n}.
	\end{equation*}
	Indeed, if for each $j=1,\dots,m$ we had
	\begin{equation*}
		\ell(P_j)^\alpha \left\| \chi_{P_j} f\right\|_{B, Q, \mathcal{H}_\infty^{\beta}}\le \frac{t}{c_\beta 2^n},
	\end{equation*}
	then it would follow that
	\begin{align*}
		\ell(Q)^\alpha \left\| f\right\|_{B, Q, \mathcal{H}_\infty^{\beta, Q_0}} = &  \ell(Q)^\alpha \left\| \chi_{\cup_{j=1}^m P_j}   f\right\|_{B, Q, \mathcal{H}_\infty^{\beta}}\\
		\le & c_\beta \sum_{j=1}^m \ell(Q)^\alpha \left\| \chi_{ P_j}   f\right\|_{B, Q, \mathcal{H}_\infty^{\beta, Q_0}}\\
		\le & c_\beta  \sum_{j=1}^m \ell(P_j)^\alpha \left\| \chi_{ P_j}   f\right\|_{B, Q, \mathcal{H}_\infty^{\beta, Q_0}}\\
		\le & c_\beta  \sum_{j=1}^m \frac{t}{c_\beta 2^n} \\
		= & \frac{m}{2^n } t\\
		\le & t,
	\end{align*}
	which contradicts \eqref{lemma orlicz eq1}. Therefore, we must have
	\begin{equation*}
		\ell(P_1)^\alpha \left\| \chi_{P_1} f\right\|_{B, Q, \mathcal{H}_\infty^{\beta, Q_0}}> \frac{t}{c_\beta 2^n}.
	\end{equation*}
	Using the fact that $\ell(P_1)^\beta  \le 2^\beta \ell(Q)^\beta$, we obtain
	\begin{align*}
		\frac{t}{c_\beta 2^n}< &\ell(P_1)^\alpha \left\| \chi_{P_1} f\right\|_{B, Q, \mathcal{H}_\infty^{\beta, Q_0}} \\
		\le & 2^\beta \ell(P_1)^\alpha \left\| f\right\|_{B, P_1, \mathcal{H}_\infty^{\beta, Q_0}}.
	\end{align*}
	This concludes the proof with the constant $C=\frac{1}{c_\beta 2^{n+\beta}}$.
\end{proof}

The following corollary is obtained by using the previous lemma together with an argument from \cite[Lemma 2.4]{chen2023capacitary}.

\begin{corollary}\label{cor: dyadic orlicz}
	There exists a constants $c=c(n,\beta)>0$ such that 
	\begin{equation*}
		\hcal_\infty^\beta \left( \left\lbrace x\in \R^n : \mathcal{M}_{\alpha, B, \mathcal{H}_\infty^\beta} f(x)>t \right\rbrace \right) \le 3^n \,  \hcal_\infty^\beta \left( \left\lbrace x\in \R^n : \mathcal{M}_{\alpha, B, \mathcal{H}_\infty^{\beta, Q_0}}^{\mathcal{D}} f(x) > c\, t \right\rbrace \right)
	\end{equation*}
	holds every $t>0$. 
\end{corollary}

\begin{proof}
	For $t>0$, we define
	\begin{equation*}
		E_t = \left\lbrace x\in \R^n : \mathcal{M}_{\alpha, B, \mathcal{H}_\infty^\beta} f(x)>t \right\rbrace,
	\end{equation*}
	and
	\begin{equation*}
		L_t = \left\lbrace x\in \R^n : \mathcal{M}_{\alpha, B, \mathcal{H}_\infty^{\beta, Q_0}}^{\mathcal{D}} f(x) > t \right\rbrace .
	\end{equation*}
	
	Let $x\in E_t$, then there exists a cube $Q_x$ with $x\in Q_x$ such that $\ell(Q)^\alpha \|f\|_{B, Q, \mathcal{H}_\infty^\beta}>t$, and this implies $Q_x\subseteq E_t$. By Lemma \ref{lemma dyadic orlicz}, there exists $P_x\in \mathcal{D}$ with $Q_x\subseteq 3P_x$ and such that $\ell(P_x)^\alpha \left\| f\right\|_{B, P_x, \mathcal{H}_\infty^{\beta, Q_0}}>c\, t$, where $c=c(n,\beta)>0$. Hence, $P_x\subseteq L_{c\, t}$ and
	\begin{equation*}
		E_t \subseteq \bigcup_{x\in E_t} Q_x \subseteq \bigcup_{x\in E_t} 3P_x \subseteq \left\lbrace 3P : P\in \mathcal{D} \text{ and } P\subseteq L_{c\, t} \right\rbrace. 
	\end{equation*}
	
	For each $P\in \mathcal{D}$ with $P\subseteq L_{c\, t}$ we can write $3P = \bigcup_{i=1}^{3^n} \tau_i P$ for appropiate translation operators $\{ \tau_i \}_{i=1}^{3^n}$. Hence, using the subadditivity and translation invariance of $\hcal_\infty^\beta$, we have
	\begin{align*}
		\hcal_\infty^\beta \left( E_t \right) \le & \hcal_\infty^\beta \left(\left\lbrace 3P : P\in \mathcal{D} \text{ and } P\subseteq L_{c\, t} \right\rbrace \right)\\
		= & \hcal_\infty^\beta \left(\left\lbrace \cup_{i=1}^{3^n} \tau_i P : P\in \mathcal{D} \text{ and } P\subseteq L_{c\, t} \right\rbrace \right)\\
		\le & \sum_{i=1}^{3^n} \hcal_\infty^\beta \left(\left\lbrace  \tau_i P : P\in \mathcal{D} \text{ and } P\subseteq L_{c\, t} \right\rbrace \right)\\
		= & 3^n \hcal_\infty^\beta \left(\left\lbrace  P : P\in \mathcal{D} \text{ and } P\subseteq L_{c\, t} \right\rbrace \right)\\
		= & 3^n \hcal_\infty^\beta \left( L_{c\, t}  \right).
	\end{align*}	
\end{proof}

\begin{definition}\label{hbfunction}
Given a Young function $B$, we define the function $h_B$ by
\begin{equation*}
h_B(s) = \sup_{t>0} \frac{B(st)}{B(t)}
\end{equation*}
for each $0 \le s < \infty$.
\end{definition}

\begin{remark}
As shown in \cite[Lemma 3.11]{Cruz-Uribe-Fiorenza}, the function $h_B$ is submultiplicative, increasing on $[0,\infty)$, and strictly increasing on $[0,1]$, with $h_B(1)=1$.
\end{remark}

We also use the following lemma, whose proof can be found in \cite{HLP}.

\begin{lemma}\label{lemma Phi}
Suppose that the function $t \mapsto \Phi(t)/t$ is decreasing. Then, for any sequence of positive numbers $\{x_k\}$, we have
\begin{equation*}
\Phi \left( \sum_k x_k \right) \le \sum_k \Phi(x_k).
\end{equation*}
\end{lemma}

With the preceding results at hand, we can now establish the following theorem, which will be essential for deriving the modular weak-type endpoint estimate stated in Theorem \ref{thm: modular weak}.

\begin{theorem}\label{thm: weak orlicz fractional}
Let $0<\beta\le n$ and $0 \le \alpha < \beta$. Let $B$ be a Young function such that $t \mapsto B(t)/t^{\beta/\alpha}$ is decreasing for all $t>0$. Then, there exists a constant $C=C(n, \beta, \alpha, B )>0$ such that for all $t>0$, the operator $\mathcal{M}_{\alpha,B,\mathcal{H}_\infty^{\beta}}$ satisfies the modular weak-type inequality
\begin{equation*}
	\Phi\left( \mathcal{H}_\infty^{\beta} \big( \{ x \in \mathbb{R}^n : \mathcal{M}_{\alpha,B,\mathcal{H}_\infty^{\beta}} f(x) > t \} \big) \right) \le C \int_{\R^n} B \left( \frac{f(x)}{t} \right) d\mathcal{H}_\infty^{\beta},
\end{equation*}
for all non-negative functions $f \in L^B(\R^n, \mathcal{H}_\infty^{\beta})$. Here, $\Phi$ denotes any function satisfying
\begin{equation*}
	\Phi(s) \le C_1 \Phi_1(s) =
	\begin{cases}
	0 & \text{if } s=0,\\[4pt]
	\dfrac{s}{h_B(s^{\alpha/\beta})} & \text{if } s>0.
	\end{cases}
\end{equation*}
\end{theorem}

\begin{remark}
The function $\Phi_1$ is well defined; indeed, by \cite[Lemma 3.12]{Cruz-Uribe-Fiorenza}, since $B(t)/t^{\beta/\alpha}$ is decreasing, it follows that $0 < h_B(s^{\alpha/\beta}) < \infty$ for all $s>0$. Furthermore, the function $\Phi_1$ is increasing and $\Phi_1(t)/t$ is decreasing. Moreover, there exists an invertible function $\Phi$ such that $\Phi(s)\le C_1 \Phi_1(s)$.
\end{remark}

\begin{proof}
Fix $t>0$, and define the sets
\begin{equation*}
E_t = \{ x \in \mathbb{R}^n : \mathcal{M}_{\alpha,B,\mathcal{H}_\infty^{\beta}} f(x) > t \},
\end{equation*}
and
\begin{equation*}
		L_t = \left\lbrace x\in \R^n : \mathcal{M}_{\alpha, B, \mathcal{H}_\infty^{\beta, Q_0}}^{\mathcal{D}} f(x) > t \right\rbrace .
\end{equation*}

By Corollary \ref{cor: dyadic orlicz} we have
\begin{equation*}
	\hcal_\infty^\beta \left( E_t \right)\le 3^n \hcal_\infty^\beta \left( L_{A t} \right),
\end{equation*}	
where $A = \frac{1}{c_{n,\beta} 3^n}$. Using that $\Phi_1$ is increasing and \eqref{equivalenceofdyadic}, we have
\begin{equation*}
\Phi_1(\mathcal{H}_\infty^{\beta}(E_t))
\le \Phi_1(C_{\beta,n} \mathcal{H}_\infty^{\beta,Q_0}(L_{At})).
\end{equation*}
We may write
\begin{equation*}
L_{At} = \bigcup_j P_j,
\end{equation*}
where $\{P_j\}_j$ is the maximal collection of disjoint dyadic cubes contained in $L_{At}$. Observe that each $P_j$ satisfies 
\begin{equation*}
\ell(P_j)^\alpha \|f\|_{B, P_j, \mathcal{H}_\infty^{\beta,Q_0}} > A t.
\end{equation*}
By the definition of the Luxemburg norm, this implies
\begin{equation*}
\|(A t)^{-1} \ell(P_j)^\alpha f\|_{B, P_j, \mathcal{H}_\infty^{\beta,Q_0}} > 1
\end{equation*}
for each $j$. Consequently,
\begin{align*}
1 < {} & \frac{1}{\ell(P_j)^\beta} \int_{P_j} 
B \left( \frac{\ell(P_j)^\alpha f(x)}{A t} \right) d\mathcal{H}_\infty^{\beta,Q_0} \\
\le {} & \frac{h_B(A^{-1}\ell(P_j)^\alpha)}{\ell(P_j)^\beta}
\int_{P_j} B \left( \frac{f(x)}{t} \right) d\mathcal{H}_\infty^{\beta,Q_0} \\
\le {} & h_B(A^{-1}C_\beta^{-\alpha/\beta})
\frac{h_B((C_\beta \ell(P_j)^{\beta})^{\alpha/\beta})}{\ell(P_j)^\beta}
\int_{P_j} B \left( \frac{f(x)}{t} \right) d\mathcal{H}_\infty^{\beta,Q_0} \\
= {} & h_B(A^{-1}C_\beta^{-\alpha/\beta})
\frac{1}{\Phi_1(C_\beta \ell(P_j)^\beta)}
\int_{P_j} B \left( \frac{f(x)}{t} \right) d\mathcal{H}_\infty^{\beta,Q_0},
\end{align*}
where $C_\beta$ is the constant from \eqref{equivalenceofdyadic}. Thus,
\begin{equation}\label{proof orlicz maximal key estimate}
\Phi_1(C_\beta \ell(P_j)^\beta)
< h_B(A^{-1}C_\beta^{-\alpha/\beta})
\int_{P_j} B \left( \frac{f(x)}{t} \right) d\mathcal{H}_\infty^{\beta,Q_0}
\end{equation}
for each $j$.

Invoking Lemma~\ref{packing2}, we obtain a subfamily $\{P_{j_v}\}_v$ of $\{P_j\}$ and a family of non-overlapping ancestors $\tilde{P}_v$ satisfying:

\begin{enumerate}
\item
\begin{equation*}
E_t^k = \bigcup_j P_j \subset \bigcup_v P_{j_v} \cup \bigcup_v \tilde{P}_v.
\end{equation*}
\item
\begin{equation*}
\sum_{P_{j_v}\subset P} \ell(P_{j_v})^\beta \le 2 \ell(P)^\beta,
\quad \text{for each dyadic cube } P \text{ (with respect to } \mathcal{D}_k).
\end{equation*}
\item
\begin{align*}
\mathcal{H}_\infty^{\beta,Q_0} \left( \bigcup_j P_j \right)
&\le \sum_v \ell(P_{j_v})^\beta + \sum_v \ell(\tilde{P}_v)^\beta \\
&\le \sum_v \ell(P_{j_v})^\beta.
\end{align*}
\end{enumerate}

Therefore, combining Lemma~\ref{lemma Phi} and \eqref{proof orlicz maximal key estimate}, we deduce that
\begin{align*}
\Phi_1(C_\beta \mathcal{H}_\infty^{\beta,Q_0}(L_{At}))
&= \Phi_1(C_\beta \mathcal{H}_\infty^{\beta,Q_0}(\cup_j P_j)) \\
&\le \Phi_1 \left( \sum_v C_\beta \ell(P_{j_v})^\beta \right)\\
&\le \sum_v \Phi_1(C_\beta \ell(P_{j_v})^\beta) \\
&\le h_B(A^{-1}C_\beta^{-\alpha/\beta})
\sum_v \int_{P_{j_v}} B \left( \frac{f(x)}{t} \right) d\mathcal{H}_\infty^{\beta,Q_0} \\
&\le 2\, h_B(A^{-1}C_\beta^{-\alpha/\beta})
\int_{\R^n} B \left( \frac{f(x)}{t} \right) d\mathcal{H}_\infty^{\beta,Q_0},
\end{align*}
where the last inequality follows from the packing condition satisfied by $\{P_{j_v}\}_v$ (see \eqref{packingestimate}). We conclude:
\begin{align*}
\Phi_1(\mathcal{H}_\infty^{\beta}(E_t))
&\le \Phi_1(C_{\beta,n} \mathcal{H}_\infty^{\beta,Q_0}(L_{At}))\\
&\le  2\, h_B(A^{-1}C_\beta^{-\alpha/\beta})
\int_{\R^n} B \left( \frac{f(x)}{t} \right) d\mathcal{H}_\infty^{\beta,Q_0} \\
&\le C_\beta  h_B(A^{-1}C_\beta^{-\alpha/\beta})
\int_{\R^n} B \left( \frac{f(x)}{t} \right) d\mathcal{H}_\infty^{\beta}.
\end{align*}
\end{proof}

\section{Riesz potential and Hausdorff content}

In this section, we state some known results concerning the Riesz potential and the fractional maximal function defined with respect to the Hausdorff content. We also prove some auxiliary results that will be needed later in the proof of Theorem \ref{thm: key pointwise}. 

The first result we will need is the boundedness of the Riesz potential $I_\alpha$ in Lebesgue spaces with respect to the Hausdorff content $\mathcal{H}^\beta_\infty$, established in \cite[Theorem 5.4]{harjulehto2024hausdorff}. 

\begin{theorem}\label{HHestimate}
	Let $0<\beta \le n\in \mathbb{N}$, $\alpha \in (0,\beta)$, and $p\in (\tfrac{\beta}{n} , \tfrac{\beta}{\alpha})$ be given. Then, for all measurable functions $f: \R^n \longrightarrow [-\infty, \infty]$, we have
	\begin{equation*}
		\left( \int_{\R^n} |I_\alpha f(x)|^{\frac{\beta p}{\beta-p\alpha}} d\mathcal{H}^\beta_\infty \right) ^ \frac{\beta-p\alpha}{\beta p} \le C \left( \int_{\R^n} |f(x)|^p d\mathcal{H}^\beta_\infty \right)^\frac{1}{p},
	\end{equation*} 
	where $C=C(n, \beta, \alpha, p)>0$. 
\end{theorem}

We will also require the boundedness of the fractional maximal operator $\mathcal{M}_{\alpha, \mathcal{H}^\beta_\infty}$. This follows readily from \cite[Theorem 5.11]{harjulehto2024hausdorff} via the pointwise estimate $\mathcal{M}_{\alpha, \mathcal{H}^\beta_\infty} f (x) \lesssim \mathcal{R}^\beta_\alpha f (x ) $, where $\mathcal{R}^\beta_\alpha f (x ):=  \int_{\R^n} \frac{f(y)}{|x-y|^{n-\alpha  }} \;d \mathcal{H}^\beta_\infty$ denotes the $\beta$-dimensional Riesz potential of order $\alpha$.

\begin{theorem}\label{HHestimate2}
	Let $0<\beta \le n\in \mathbb{N}$, $\alpha \in (0,\beta)$, and $p\in (1 , \tfrac{\beta}{\alpha})$ be given. Then, there exists a constant $C=C(n,\beta ,\alpha, p)>0$ such that 
	\begin{equation*}
		\left( \int_{\R^n} \mathcal{M}_{\alpha, \mathcal{H}^\beta_\infty} f(x)^{\frac{\beta p}{\beta-p\alpha}} d\mathcal{H}^\beta_\infty \right) ^ \frac{\beta-p\alpha}{\beta p} \le C \left( \int_{\R^n} |f(x)|^p d\mathcal{H}^\beta_\infty \right)^\frac{1}{p},
	\end{equation*} 
	 for every function $f: \R^n \longrightarrow [-\infty, \infty]$. 
\end{theorem}

The first auxiliary result is a modification of \cite[Lemma 3.5]{chen2024selfimproving}.

\begin{lemma}\label{forcompactmorreya}
Let $0< \alpha< \beta \leq n \in \mathbb{N}$, let $Q$ be an open cube in $\mathbb{R}^n$ with center $x_0$, and let $f$ be a measurable function defined on $\mathbb{R}^n$. If $\supp(f) \subseteq 2Q$, then there exists a constant $C=C(n,\alpha, \beta)>0$ such that
\begin{align*}
    \int_Q | I_\alpha f |\, d\mathcal{H}^{\beta}_\infty \leq C \ell(Q)^{\alpha} \int_{2Q}|f|\, d\mathcal{H}_\infty^\beta.
\end{align*}

\end{lemma}

\begin{proof}
First, recall that for any $\eta \in [0,\alpha)$, there exists a constant $C_\eta>0$ such that
\begin{align}\label{Lemma45X1}
|I_\alpha \phi(x)| \leq \frac{C_\eta}{\gamma(\alpha)}\ell(Q)^{\alpha-\eta}\mathcal{M}_\eta \phi(x),
\end{align}
which can be found, for instance, in \cite[(3.6)]{chen2024selfimproving}. Additionally, by \cite[Theorem 5.4]{harjulehto2024hausdorff}, there exists a constant $C>0$ such that
\begin{align}\label{Lemma45X2}
\left(\int_{\mathbb{R}^n}(\mathcal{M}_\eta f)^{\frac{\beta}{\beta-\eta}}\, d\mathcal{H}^\beta_\infty\right)^{\frac{\beta-\eta}{\beta}}\leq C\int_{\mathbb{R}^n}|f|\, d\mathcal{H}^\beta_\infty.
\end{align}
The combination of \eqref{Lemma45X1} and \eqref{Lemma45X2} then yields that 
\begin{align*}
\int_Q |I_\alpha f|\, d\mathcal{H}^\beta_\infty &\leq C\ell(Q)^{\alpha - \eta}\int_Q\mathcal{M}_\eta f\, d\mathcal{H}^\beta_\infty \\
&\leq C\ell(Q)^{\alpha - \eta}\left(\int_Q(\mathcal{M}_\eta f)^{\frac{\beta}{\beta-\eta}}\, d\mathcal{H}^\beta_\infty\right)^{\frac{\beta-\eta}{\beta}}\left(\mathcal{H}^\beta_\infty(Q)\right)^{\frac{\eta}{\beta}} \\
&\leq C\ell(Q)^\alpha\int_{2Q}|f|\, d\mathcal{H}^\beta_\infty.
\end{align*}
This completes the proof.
\end{proof}

The second auxiliary result we need is a modification of \cite[Lemma 3.8]{chen2024selfimproving}.

\begin{lemma}\label{outsidef2inMorrey}
    Let $0<\alpha < n \in \mathbb{N}$, $\beta \in (0,n]$, $Q$ be an open cube in $\mathbb{R}^n$ with centre $x_0$ and $f$ be a  measurable function with $\supp{f} \subseteq (2Q)^c$. Then, there exists $c \in \mathbb{R}$, and $C=C(n,\beta,\alpha)>0$ such that 
    \begin{align*}
\frac{1}{\ell(Q)^\beta}\int_{Q} | I_\alpha f -c| \, d\mathcal{H}^\beta_\infty \leq C \sum_{k=0}^\infty 2^{-k} (2^{k+1}\ell(Q))^{\alpha} \frac{1}{( 2^{k+1} \ell( Q))^{\beta  } } \int_{2^{k+1}  Q} |f| d\mathcal{H}_\infty^\beta.     
\end{align*}

\end{lemma}

\begin{proof}
	It is enough to prove the result for a cube $Q$ with centre $x_0=0$. We follow the proof of \cite[Lemma 3.8]{chen2024selfimproving}, taking $c=I_\alpha f(0)$ until (3.24),
	\begin{align*}
	\nonumber	\frac{1}{\ell(Q)^\beta} \int_{Q} | I_\alpha f -c|  d\mathcal{H}^\beta_\infty \le & \frac{C}{\ell(Q)^\beta} \int_{Q} \int_{(2Q)^c} \left| \frac{1}{|x-y|^{n-\alpha} } - \frac{1}{|y|^{n-\alpha}} \right| |f(y)|dy  d\mathcal{H}^\beta_\infty(x) \\
\nonumber		\le & \frac{C}{\ell(Q)^\beta} \int_Q |x|  d\mathcal{H}^\beta_\infty(x)  \int_{(2Q)^c} \frac{|f(y)|}{|y|^{n-\alpha+1}} dy\\
\nonumber		\le & \frac{C}{\ell(Q)^\beta} \ell(Q)^{\beta+1} \int_{|y|\ge \ell(Q)} \frac{|f(y)|}{|y|^{n-\alpha+1}} dy\\
\nonumber		= & C \ell(Q) \sum_{k=0}^\infty \int_{2^{k-1}\ell(Q)\le |y|< 2^k  \ell(Q)} \frac{|f(y)|}{|y|^{n-\alpha+1}} dy\\
\nonumber		\le & C \ell(Q) \sum_{k=0}^\infty (2^{k-1}\ell(Q))^{-n+\alpha-1} \int_{ |y|< 2^k  \ell(Q)} |f(y)| dy\\
\nonumber		\le & C  \sum_{k=0}^\infty 2^{-k} (2^{k-1}\ell(Q))^{-n+\alpha} \int_{ 2^{k+1}  Q} |f(y)| dy\\
		= & C'  \sum_{k=0}^\infty 2^{-k} (2^{k+1}\ell(Q))^{-n+\alpha} \int_{ 2^{k+1}  Q} |f(y)| dy
	\end{align*}

	If $\beta=n$, we have concluded the proof. Let us assume $\beta<n$, then using Lemma \ref{alphabetachange},  and Holder inequality (iii) of Lemma \ref{basicChoquet}, we have
	\begin{align*}
		\int_{ 2^{k+1}  Q} |f(y)| dy \le & \frac{n}{\beta} \left( \int_{ 2^{k+1}  Q} |f|^\frac{\beta}{n} d\mathcal{H}_\infty^\beta   \right)^\frac{n}{\beta}\\
		\le &  2\frac{n}{\beta } \left(  \int_{2^{k+1}  Q} |f| d\mathcal{H}_\infty^\beta \right) \ell(2^{k+1}  Q)^{\beta \frac{\frac{n}{\beta}}{(\frac{n}{\beta})'} }\\
		= &2\frac{n}{\beta } ( 2^{k+1} \ell( Q))^{n-\beta  }  \int_{2^{k+1}  Q} |f| d\mathcal{H}_\infty^\beta  .
	\end{align*}
	Hence, 
	\begin{align*}
		\int_{Q} | I_\alpha f -c|  d\mathcal{H}^\beta_\infty \le & C   \sum_{k=0}^\infty 2^{-k} (2^{k+1}\ell(Q))^{-n+\alpha} \int_{ 2^{k+1}  Q} |f(y)| dy\\
		\le & C' \sum_{k=0}^\infty 2^{-k} (2^{k+1}\ell(Q))^{-n+\alpha} ( 2^{k+1} \ell( Q))^{n-\beta  }  \int_{2^{k+1}  Q} |f| d\mathcal{H}_\infty^\beta  \\
		= & C' \sum_{k=0}^\infty 2^{-k} (2^{k+1}\ell(Q))^{\alpha} \frac{1}{( 2^{k+1} \ell( Q))^{\beta  } } \int_{2^{k+1}  Q} |f| d\mathcal{H}_\infty^\beta.
	\end{align*}
\end{proof}

\begin{lemma}\label{lem:Ialpha A1}
    Let $0< \alpha < n$ and $0 \leq \beta < n$. Then 
\begin{align*}
    \mathcal{M}_{\cal{H}_\infty^\beta} (I_\alpha f ) (x) \leq  (\mathcal{M}_{\cal{H}_\infty^\beta} I_\alpha )*  |f| (x) 
\end{align*}
for any $x  \in \mathbb{R}^n$, where $I_\alpha(x)=|x|^{\alpha-n}$. In addition, if $\beta\in (n-\alpha, n]$, we have
\begin{align*}
    \mathcal{M}_{\cal{H}_\infty^\beta} (I_\alpha f ) (x) \leq  C \, I_\alpha ( |f|) (x) 
\end{align*}
for $\mathcal{H}_\infty^\beta$-a.e. $x\in \R^n$. That is, if $\beta\in (n-\alpha, n]$ and $f$ is nonnegative function, then $I_\alpha f\in A_1^\beta$.  
\end{lemma}

\begin{proof}
    We may assume without loss of generality that $f$ is nonnegative. Fix $z \in \mathbb{R}^n$ and a cube $Q$ containing $z$. Since $I_\alpha f$ is lower semicontinuous, we may apply \cite[Corollary on p. 118]{AdamsChoquet}. In particular, recalling the Morrey space
\begin{align*}
    \mathcal{M}^{\beta}( \mathbb{R}^n):= \left\{ \mu \in M_{loc}( \mathbb{R}^n): \|\mu\|_{\mathcal{M}^\beta ( \mathbb{R}^n)} := \sup_{x \in  \mathbb{R}^n,r>0} \frac{|\mu|(B(x,r) )}{r^\beta}<\infty \right\},
\end{align*}
we obtain
\begin{align}\label{fubiniestimate1}
  \nonumber  \frac{1}{\cal{H}_\infty^\beta(Q)} \int_Q I_\alpha f (x) \, d \cal{H}_\infty^\beta & = \sup_{\|\mu\|_{\mathcal{M}^\beta ( \mathbb{R}^n)}\le 1} \frac{1}{\cal{H}_\infty^\beta(Q)} \int_Q I_\alpha f (x) \, d \mu (x )\\
    &=  \sup_{\|\mu\|_{\mathcal{M}^\beta ( \mathbb{R}^n)}\le 1} \frac{1}{\cal{H}_\infty^\beta(Q)} \int_Q    \int_{\mathbb{R}^n} \frac{f(y)}{|x-y|^{n-\alpha }}\,dy   \, d \mu (x )\\
     &=  \sup_{\|\mu\|_{\mathcal{M}^\beta ( \mathbb{R}^n)}\le 1} \frac{1}{\cal{H}_\infty^\beta(Q)} \int_Q    \int_{\mathbb{R}^n} \frac{f(y)}{|x-y|^{n-\alpha }}\,dy   \, d \mu (x )\\
     & = \sup_{\|\mu\|_{\mathcal{M}^\beta ( \mathbb{R}^n)}\le 1} \int_{\mathbb{R}^n} f(y) \left(  \frac{1}{\cal{H}_\infty^\beta(Q)}  \int_{Q} \frac{d \mu (x )}{|x-y|^{n-\alpha }}   \right)\,dy .
\end{align}
where we have used Fubini's theorem. As $\mu(E) \leq \mathcal{H}^\beta_\infty (E)$ for any measurable set $E\subseteq \mathbb{R}^n$, we have 
\begin{align}\label{fubiniestimate2}
    \nonumber  \frac{1}{\cal{H}_\infty^\beta(Q)}  \int_{Q} \frac{d \mu (x )}{|x-y|^{n-\alpha }}   &\leq \frac{1}{\cal{H}_\infty^\beta(Q)}  \int_{Q} \frac{1}{|x-y|^{n-\alpha }}  \, d \mathcal{H}^\beta_\infty (x)\\
    & \leq \mathcal{M}_{\mathcal{H}^\beta_\infty} (I_\alpha ) (z-y).
\end{align}
The combination of (\ref{fubiniestimate1}) and (\ref{fubiniestimate2}) then yields that 
\begin{align*}
    \frac{1}{\cal{H}_\infty^\beta(Q)} \int_Q I_\alpha f (x) \, d \cal{H}_\infty^\beta(x) &\leq \int_{\mathbb{R}^n } f(y )  \, \mathcal{M}_{\mathcal{H}^\beta_\infty} (I_\alpha ) (z-y) \, dy \\
    & =  (\mathcal{M}_{\cal{H}_\infty^\beta} I_\alpha )*  f (z).
\end{align*}
Taking the supremum over all such cubes $Q$, we obtain 
\begin{align*}
    \mathcal{M}_{\cal{H}_\infty^\beta} (I_\alpha f ) (z) \leq  (\mathcal{M}_{\cal{H}_\infty^\beta} I_\alpha )*  f (z).
\end{align*}
To conclude the proof, we use the fact (see \cite[Proposition 2.14]{HZZ}) that if 
$\beta \in (n-\alpha, n]$, then $I_\alpha$ is an $A_1^\beta$ weight; that is,
\begin{equation*}
    \mathcal{M}_{\mathcal{H}_\infty^\beta} I_\alpha(x) \le C\, I_\alpha(x)
\end{equation*}
for $\mathcal{H}_\infty^\beta$-a.e. $x\in \R^n$. This, together with the previous estimate, completes the proof. 
\end{proof}

We now state the reverse Hölder inequality from \cite[Theorem 1.7]{HZZ} in the particular case of $A_1^\beta$ weights, which will be needed in what follows.

\begin{theorem}\label{thm: RHI A_1}
    Let $0<\beta\le n$ and $w\in A_1^\beta$ (that is, $\mathcal{M}_{\mathcal{H}_\infty^\beta} w(x)\le C w(x)$ for $\mathcal{H}_\infty^\beta$-a.e. $x$). There exist $C>0$ and $r>0$ such that 
    \begin{equation*}
        \left( \frac{1}{\ell(Q)^\beta }\int_Q w^{1+r} \, d\hcal_\infty^\beta \right)^\frac{1}{1+r} \le C \frac{1}{\ell(Q)^\beta} \int_Q w \, d\hcal_\infty^\beta 
    \end{equation*}
    for every cube $Q$.
\end{theorem}

\section{Key pointwise estimate}

In this section, we present the proof of Theorem \ref{thm: key pointwise}, which will play a central role in establishing the remaining main results of the paper.

\begin{proof}[Proof of Theorem \ref{thm: key pointwise}]
By homogeneity, it suffices to establish \eqref{key pointwise estimate} for $x=0$. Fix a cube $Q$ centered at the origin. We must show that there exists a constant $c\in \R$ such that
\begin{equation*}
	\frac{1}{\ell(Q)^\beta} \int_Q |[b,I_\alpha]f-c|d\mathcal{H}_\infty^\beta \le C \| b\|_{\operatorname{BMO}^\beta (\R^n) } \left( \mathcal{M}_{\mathcal{H}_\infty^\beta} ((I_\alpha f)^s) (0)^\frac{1}{s} + \mathcal{M}_{\alpha, B, \mathcal{H}_\infty^\beta} f(0) \right)
\end{equation*}
for all $s>1$. 

We decompose $f=f_1+f_2$, where $f_1=f\chi_{Q^*}$ and $Q^*=2Q$. Since $b_{Q^*}$ is constant, we observe that $[b,I_\alpha]f = [b-b_{Q^*},I_\alpha]f$. Consequently, for a constant $c_Q\in \R$ to be chosen later, we have
\begin{align*}
	\frac{1}{\ell(Q)^\beta} \int_Q |[b,I_\alpha]f-c_Q|d\mathcal{H}_\infty^\beta   = & \frac{1}{\ell(Q)^\beta} \int_Q |[b-b_{Q^*},I_\alpha]f-c_Q|d\mathcal{H}_\infty^\beta \\
	= & \frac{1}{\ell(Q)^\beta} \int_Q | (b(x)-b_{Q^*})I_\alpha f(x) - I_\alpha ((b-b_{Q^*})f)(x)  -c_Q|d\mathcal{H}_\infty^\beta \\
	\le & \frac{1}{\ell(Q)^\beta} \int_Q | (b(x)-b_{Q^*})I_\alpha f(x) |d\mathcal{H}_\infty^\beta \\
	& + \frac{1}{\ell(Q)^\beta} \int_Q |  I_\alpha ((b-b_{Q^*})f_1)(x)|d\mathcal{H}_\infty^\beta \\
	& + \frac{1}{\ell(Q)^\beta} \int_Q |I_\alpha ((b-b_{Q^*})f_2)(x)  + c_Q|d\mathcal{H}_\infty^\beta \\
	= & I_1 + I_2 + I_3.
\end{align*}

We estimate each term separately. Let $s>1$. By applying Hölder's inequality (iii) from Theorem \ref{basicChoquet} and Theorem \ref{thm: JN p}, we obtain
\begin{align}\label{estimateI10}
\nonumber	I_1 \le & 2 \left( \frac{1}{\ell(Q)^\beta} \int_Q | b(x)-b_{Q^*} |^{s'} d\mathcal{H}_\infty^\beta \right)^\frac{1}{s'} \left( \frac{1}{\ell(Q)^\beta} \int_Q |I_\alpha f(x) |^s d\mathcal{H}_\infty^\beta  \right)^\frac{1}{s}\\
\nonumber	\le & 2^{1+\frac{\beta}{s'}} \left( \frac{1}{\ell(Q^*)^\beta} \int_{Q^*} | b(x)-b_{Q^*} |^{s'} d\mathcal{H}_\infty^\beta \right)^\frac{1}{s'} \left( \frac{1}{\ell(Q)^\beta} \int_Q |I_\alpha f(x) |^s d\mathcal{H}_\infty^\beta  \right)^\frac{1}{s}\\
\nonumber	\le & 2^{1+\frac{\beta}{s'}} C s' \| b\|_{\operatorname{BMO}^\beta (\R^n) }\left( \frac{1}{\ell(Q)^\beta} \int_Q |I_\alpha f(x) |^s d\mathcal{H}_\infty^\beta  \right)^\frac{1}{s}\\
	\le & 2^{1+\frac{\beta}{s'}} C s' \| b\|_{\operatorname{BMO}^\beta (\R^n) }\mathcal{M}_{\mathcal{H}_\infty^\beta} ((I_\alpha f)^s) (0)^\frac{1}{s}.
\end{align}

To estimate $I_2$, we apply Lemma \ref{forcompactmorreya} to the function $(b-b_{Q^*})f_1$, followed by Hölder's inequality \eqref{Holder Orlicz} with $B(t)=t\log(e+t)$:
\begin{align*}
	I_2 \le & C \ell(Q^*)^\alpha \frac{1}{\ell(Q^*)^\beta}\int_{Q^*} |(b-b_{Q^*})f_1| d\mathcal{H}_\infty^\beta \\
	 = & C \ell(Q^*)^\alpha \frac{1}{\ell(Q^*)^\beta}\int_{Q^*} |(b-b_{Q^*})f| d\mathcal{H}_\infty^\beta \\
	 \le & C_\beta \ell(Q^*)^\alpha  \| f \|_{B, Q^*, \mathcal{H}_\infty^\beta}  \| b-b_{Q^*} \|_{\bar{B}, Q^*, \mathcal{H}_\infty^\beta}\\
	 \le &  C_\beta  \mathcal{M}_{\alpha, B, \mathcal{H}_\infty^\beta} f(0)  \| b-b_{Q^*} \|_{\bar{B}, Q^*, \mathcal{H}_\infty^\beta}\\
	 \le & C_\beta  \mathcal{M}_{\alpha, B, \mathcal{H}_\infty^\beta} f(0) \| b\|_{\operatorname{BMO}^\beta (\R^n) },
\end{align*}
where in the last inequality we have used Remark \ref{rem: B complementary}, Theorem \ref{exp BMO} and Remark \ref{rem exp BMO}. 

We now choose $c_Q= -c$, where $c\in \R$ is the constant obtained by applying Lemma \ref{outsidef2inMorrey} to the function $(b-b_{Q^*})f_2$. Using Lemma \ref{outsidef2inMorrey}, we have
\begin{align*}
	I_3 = & \frac{1}{\ell(Q)^\beta} \int_Q |I_\alpha ((b-b_{Q^*})f_2)(x)  -c|d\mathcal{H}_\infty^\beta\\
	\le & C \sum_{k=0}^\infty 2^{-k} (2^{k+1}\ell(Q))^{\alpha} \frac{1}{( 2^{k+1} \ell( Q))^{\beta  } } \int_{2^{k+1}  Q} |(b-b_{Q^*})f_2| d\mathcal{H}_\infty^\beta\\
	= &  C \sum_{k=0}^\infty 2^{-k} (2^{k+1}\ell(Q))^{\alpha} \frac{1}{( 2^{k+1} \ell( Q))^{\beta  } } \int_{2^{k+1}  Q} |(b-b_{2^{k+1}Q})f_2 + (b_{2^{k+1}Q} -b_{Q^*})f_2 | d\mathcal{H}_\infty^\beta\\
	\le & C_\beta \sum_{k=0}^\infty 2^{-k} (2^{k+1}\ell(Q))^{\alpha} \frac{1}{( 2^{k+1} \ell( Q))^{\beta  } } \int_{2^{k+1}  Q} |b-b_{2^{k+1}Q}||f|  d\mathcal{H}_\infty^\beta\\
	& + C_\beta \sum_{k=0}^\infty 2^{-k}  |b_{2^{k+1}Q} -b_{Q^*} | (2^{k+1}\ell(Q))^{\alpha} \frac{1}{( 2^{k+1} \ell( Q))^{\beta  } } \int_{2^{k+1}  Q} |f|  d\mathcal{H}_\infty^\beta\\
	= & I_{31} + I_{32}.
\end{align*}

To estimate $I_{31}$, we repeat the argument used for $I_2$:
\begin{align*}
	I_{31} \le & C_\beta \sum_{k=0}^\infty 2^{-k} (2^{k+1}\ell(Q))^{\alpha} \| f \|_{B, 2^{k+1}  Q, \mathcal{H}_\infty^\beta}  \| b-b_{2^{k+1}  Q} \|_{\bar{B}, 2^{k+1}  Q, \mathcal{H}_\infty^\beta}\\
	\le & C_\beta \mathcal{M}_{\alpha, B, \mathcal{H}_\infty^\beta} f(0)\| b\|_{\operatorname{BMO}^\beta (\R^n) }\sum_{k=0}^\infty 2^{-k} \\
	= & C_\beta \mathcal{M}_{\alpha, B, \mathcal{H}_\infty^\beta} f(0)\| b\|_{\operatorname{BMO}^\beta (\R^n) }.
\end{align*}

To estimate $I_{32}$, we invoke Lemma \ref{bmo cQ} and Remark \ref{remark Holder}:
\begin{align*}
	I_{32} = &  C_\beta \sum_{k=0}^\infty 2^{-k}  |b_{2^{k}Q^*} -b_{Q^*} | (2^{k+1}\ell(Q))^{\alpha} \frac{1}{( 2^{k+1} \ell( Q))^{\beta  } } \int_{2^{k+1}  Q} |f|  d\mathcal{H}_\infty^\beta\\
	\le & C_\beta \sum_{k=0}^\infty 2^{-k}  k \| b\|_{\operatorname{BMO}^\beta (\R^n) } (2^{k+1}\ell(Q))^{\alpha} \frac{1}{( 2^{k+1} \ell( Q))^{\beta  } } \int_{2^{k+1}  Q} |f|  d\mathcal{H}_\infty^\beta\\
	\le & C_\beta \| b\|_{\operatorname{BMO}^\beta (\R^n) }  \mathcal{M} _{\alpha, \mathcal{H}_\infty^\beta} f(0) \sum_{k=0}^\infty 2^{-k}  k  \\
	\le & C_\beta \| b\|_{\operatorname{BMO}^\beta (\R^n) }  \mathcal{M} _{\alpha, B,  \mathcal{H}_\infty^\beta} f(0).
\end{align*}

This concludes the proof of the first statement. Now, assume that $\beta\in (n-\alpha, n]$. Hence, by Lemma \ref{lem:Ialpha A1} we have $I_\alpha f\in A_1^\beta$. Returning to \eqref{estimateI10}, we choose $s>1$ to be the exponent provided by the reverse Hölder inequality (Theorem \ref{thm: RHI A_1}). We then obtain
\begin{align*}
   I_1 \le & 2^{1+\frac{\beta}{s'}} C s' \| b\|_{\operatorname{BMO}^\beta (\R^n) }\left( \frac{1}{\ell(Q)^\beta} \int_Q |I_\alpha f(x) |^s d\mathcal{H}_\infty^\beta  \right)^\frac{1}{s}\\ 
    \le & 2^{1+\frac{\beta}{s'}} C s' \| b\|_{\operatorname{BMO}^\beta (\R^n) } \frac{1}{\ell(Q)^\beta} \int_Q |I_\alpha f(x) | d\mathcal{H}_\infty^\beta  \\
    \le & 2^{1+\frac{\beta}{s'}} C s' \| b\|_{\operatorname{BMO}^\beta (\R^n) }I_\alpha f(0),
\end{align*}
where in the last inequality we have used Corollary \ref{lem:Ialpha A1} again. This establishes \eqref{key pointwise estimate 2}.
\end{proof}

\begin{remark}\label{Rmk pointwise proof}
    We indicate how the argument in the proof of Theorem \ref{thm: key pointwise} also yields the pointwise estimate stated in Remark \ref{Rmk pointwise}, without appealing to the Orlicz fractional maximal function. Using the notation of the proof of Theorem \ref{thm: key pointwise}, it suffices to revisit the terms $I_2$ and $I_3$. Fix $t>1$. Replacing the use of Lemma \ref{lem: Holder orlicz} by the H\"older inequality stated in Theorem \ref{basicChoquet} (iii) with exponent $t$, and combining it with Theorem \ref{thm: JN p}, we obtain
\begin{equation*}
    I_2+I_3\ \lesssim\ \|b\|_{\mathrm{BMO}^\beta(\R^n)}\,
 M_{\alpha t,\mathcal{H}^\beta_\infty}(f^t)(x)^{1/t}.
\end{equation*}
Together with the estimate for $I_1$ already established in the proof of Theorem \ref{thm: key pointwise}, this yields that for every $s,t>1$,
\begin{equation*}
  	\mathcal{M}^\#_\beta ( [b, I_\alpha] f )(x) \lesssim  \| b\|_{\operatorname{BMO}^\beta (\R^n) } \left( \mathcal{M}_{\mathcal{H}_\infty^\beta} ((I_\alpha f)^s) (x)^\frac{1}{s} + \mathcal{M}_{\alpha t, \mathcal{H}_\infty^\beta} (f^t)(x)^\frac{1}{t} \right),
\end{equation*}
which is precisely the desired bound.
\end{remark}

\begin{remark}\label{rem:iteratedcommutators}
The previous proof can be extended naturally to iterated commutators. For $m\in \mathbb{N}$ and $b$ locally integrable, we define
\begin{equation*}
(I_\alpha)_b^1 f:=[b,I_\alpha]f,\qquad (I_\alpha)_b^{m}f:=\big[b,(I_\alpha)_b^{m-1}\big]f,\quad m\ge 2.
\end{equation*}
Equivalently,
\begin{equation*}
(I_\alpha)_b^m f(x)
=\int_{\R^n} \frac{(b(x)-b(y))^m}{|x-y|^{n-\alpha}}\, f(y)\,dy.
\end{equation*}
In particular, $(I_\alpha)_b^m$ is invariant under the addition of constants to $b$. The proof of Theorem~\ref{thm: key pointwise} adapts to $(I_\alpha)_b^m$ by using the binomial expansion
\begin{equation*}
(I_\alpha)_b^m f(x)
=\sum_{j=0}^m \binom{m}{j}(-1)^j\, b(x)^{m-j}\, I_\alpha\big(b^j f\big)(x),
\end{equation*}
applied with $b_0:=b-b_{Q^*}$, and treating each summand as in the case $m=1$ after splitting $f=f_1+f_2$ where $f_1=f\chi_{Q^*}$. This yields the analogue of \eqref{key pointwise estimate} with the natural dependence $\|b\|_{\operatorname{BMO}^\beta(\R^n)}^{m}$ and with the Young function $B$ replaced by $B_m(t):=t(\log(e+t))^m$. Consequently, the corresponding analogues of Theorem~\ref{Thm LpLq bound commutator} and Theorem~\ref{thm: modular weak} follow by the same argument. We do not pursue these extensions here.
\end{remark}

\section{Characterization of $\operatorname{BMO}^\beta$ in terms of commutator of fractional integrals}

In this section, we present the proof of Theorem \ref{thm: characterization BMO commutator}. For clarity, we divide the characterization into two separate results: first, we show that the condition $b \in \operatorname{BMO}^\beta$ is sufficient for the boundedness of the commutator $[b, I_\alpha]$, and later we prove that this condition is also necessary.

\begin{theorem}\label{Thm LpLq bound commutator}
	Let $0< \beta \leq n$, $0<\alpha < \beta$ and let $b$ be a measurable function. For each $1<p<\frac{\beta}{\alpha}$, we define the exponent $q$ by
	\begin{equation*}
		\frac{1}{p}-\frac{1}{q}= \frac{\alpha}{\beta}.
	\end{equation*}
	
	If $b\in \operatorname{BMO}^\beta (\R^n)$, then there exists $C=C(\alpha, \beta, n, p)>0$ such that 
	\begin{equation}\label{LpLq bound commutator}
		\left( \int_{\R^n} \left| [b, I_\alpha] f \right|^q d\mathcal{H}_\infty^\beta \right)^\frac{1}{q} \le C \|b\|_{\operatorname{BMO}^\beta} \left( \int_{\R^n} \left|f \right|^p d\mathcal{H}_\infty^\beta \right)^\frac{1}{p},
	\end{equation}
	for every bounded measurable function $f
    $ with compact support.
\end{theorem}

\begin{proof}

	Let $b\in \operatorname{BMO}^\beta (\R^n)$ and let $f$ be a bounded measurable function with compact support. We may assume without loss of generality that $\|b\|_{\operatorname{BMO}^\beta}\neq 0$. For $k>0$, we define the truncated function $b_k$ as in Lemma \ref{truncatedBMO}, hence 
    \begin{equation*}
        \|b_k\|_{\operatorname{BMO}^\beta} \le C_\beta \|b\|_{\operatorname{BMO}^\beta}.
    \end{equation*}

    Fix some $p_0\in (1,p)$, and define $q_0$ by the relation $\tfrac{1}{p_0}-\tfrac{1}{q_0}= \tfrac{\alpha}{\beta}$. Since $f$ has compact support and $b_k$ is bounded, $\|b_kf\|_{L^{p_0}(\mathcal{H}^{\beta}_{\infty})}<\infty$ and by Theorem \ref{HHestimate} we have $\|I_\alpha (b_kf)\|_{L^{q_0}(\mathcal{H}^{\beta}_{\infty})}<\infty$. Also, since $\|f\|_{L^{p_0}(\mathcal{H}^{\beta}_{\infty})}<\infty$, again by Theorem \ref{HHestimate} we get $\|I_\alpha f\|_{L^{q_0}(\mathcal{H}^{\beta}_{\infty})}<\infty$ and hence $\|b_k I_\alpha f\|_{L^{q_0}(\mathcal{H}^{\beta}_{\infty})}<\infty$. Therefore, $[b_k, I_\alpha ]f$ is well defined and we have 
    \begin{equation*}
       \|[b_k, I_\alpha ]f\|_{L^{q_0}(\mathcal{H}^{\beta}_{\infty})}  = \|b_k I_\alpha f-I_\alpha (b_k f)\|_{L^{q_0}(\mathcal{H}^{\beta}_{\infty})} <\infty. 
    \end{equation*}
    In particular, the weak-type estimate \eqref{newassumptionThm1.2} needed to apply Theorem \ref{thm: beta FS} is satisfied by $[b_k, I_\alpha ]f$. 

    Using Lebesgue differentiation theorem \cite[Theorem 2.9]{ChenClaros} and Theorem \ref{thm: beta FS}, we deduce that there exists a constant $C>0 $ depending only on $q,$ $n$, $\beta$ such that
	\begin{align*}
		\left( \int_{\R^n} \left| [b_k, I_\alpha] f \right|^q d\mathcal{H}_\infty^\beta \right)^\frac{1}{q} \le & \left( \int_{\R^n} \left|\mathcal{M}_{\mathcal{H}_\infty^\beta }( [b_k, I_\alpha] f) \right|^q d\mathcal{H}_\infty^\beta \right)^\frac{1}{q} \\
		\le & C \left( \int_{\R^n} \left|\mathcal{M}_{\beta }^\# ( [b_k, I_\alpha] f) \right|^q d\mathcal{H}_\infty^\beta \right)^\frac{1}{q}.
	\end{align*}
    Now, by using the pointwise inequality given in Theorem \ref{thm: key pointwise} (concretely the inequality given in Remark \ref{Rmk pointwise}) with $1<\lambda, t<q$, we obtain 
    \begin{align*}
        \left( \int_{\R^n} \left|\mathcal{M}_{\beta }^\# ( [b_k, I_\alpha] f) \right|^q d\mathcal{H}_\infty^\beta \right)^\frac{1}{q} 
        &\le C \|b_k\|_{\operatorname{BMO}^\beta} \left( \int_{\R^n} \left| \mathcal{M}_{\mathcal{H}_\infty^\beta} ((I_\alpha f)^s) (x)  \right|^\frac{q}{s} d\mathcal{H}_\infty^\beta \right)^\frac{1}{q}\\
        &\quad +C \|b_k\|_{\operatorname{BMO}^\beta} \left( \int_{\R^n} \left| \mathcal{M}_{\alpha t, \mathcal{H}_\infty^\beta} (f^t)(x) \right|^\frac{q}{t} d\mathcal{H}_\infty^\beta \right)^\frac{1}{q}\\
        &\le C \|b\|_{\operatorname{BMO}^\beta} \left( \int_{\R^n} \left| \mathcal{M}_{\mathcal{H}_\infty^\beta} ((I_\alpha f)^s) (x)  \right|^\frac{q}{s} d\mathcal{H}_\infty^\beta \right)^\frac{1}{q}\\
        &\quad +C \|b\|_{\operatorname{BMO}^\beta} \left( \int_{\R^n} \left| \mathcal{M}_{\alpha t, \mathcal{H}_\infty^\beta} (f^t)(x) \right|^\frac{q}{t} d\mathcal{H}_\infty^\beta \right)^\frac{1}{q}\\
        &= I_1+ I_2.
    \end{align*}

We will estimate each term separately. Firstly, (choosing $s $ such that $\frac{q}{s} >1$), using the boundedness of $\mathcal{M}_{\mathcal{H}_\infty^\beta}$ proved in \cite[Theorem 1.2]{chen2023capacitary}, we have
\begin{align*}
    I_1 &\le C \|b\|_{\operatorname{BMO}^\beta} \left( \int_{\R^n} \left| I_\alpha f \right|^q d\mathcal{H}_\infty^\beta \right)^\frac{1}{q}\\
    &\le C \|b\|_{\operatorname{BMO}^\beta} \left( \int_{\R^n} \left|f \right|^p d\mathcal{H}_\infty^\beta \right)^\frac{1}{p}.
\end{align*}
where in the last inequality we have used Theorem \ref{HHestimate}. We estimate the second term using Theorem \ref{HHestimate2}, where we use the relation
    \begin{equation*}
        \frac{t}{p}-\frac{t}{q}=\frac{t\alpha}{\beta},
    \end{equation*}
    therefore
\begin{align*}
I_2 =& C \|b_k\|_{\operatorname{BMO}^\beta} \left( \int_{\R^n} \left| \mathcal{M}_{\alpha t, \mathcal{H}_\infty^\beta} (f^t)(x) \right|^\frac{q}{t} d\mathcal{H}_\infty^\beta \right)^{\frac{t}{q} \frac{1}{t}}\\
\le& C \|b\|_{\operatorname{BMO}^\beta} \left( \int_{\R^n} \left|f \right|^{t \frac{p}{t}} d\mathcal{H}_\infty^\beta \right)^{\frac{t}{p} \frac{1}{t}}\\
=& C \|b\|_{\operatorname{BMO}^\beta} \left( \int_{\R^n} \left|f \right|^p d\mathcal{H}_\infty^\beta \right)^\frac{1}{p}.
\end{align*}

Hence, we have proved 
\begin{align*}
    \left( \int_{\R^n} \left| [b_k, I_\alpha] f \right|^q d\mathcal{H}_\infty^\beta \right)^\frac{1}{q} \le C \|b\|_{\operatorname{BMO}^\beta} \left( \int_{\R^n} \left|f \right|^p d\mathcal{H}_\infty^\beta \right)^\frac{1}{p},
\end{align*}
for every $k>0$ with constant $C$ independent on $k$. 

To conclude the proof for the general symbol $b$ we take the limit as $k\to \infty$. Since $f$ is bounded and has compact support, let us denote $K=\supp f$. By the monotone convergence theorem, we have
\begin{equation*}
    \left( \int_{\R^n} |b_k(x)f(x) - b(x)f(x)|^s  \, d\mathcal{H}_\infty^\beta \right)^{\frac{1}{s}} =\left( \int_{K} |b_k(x) - b(x)|^s |f(x)|^s \, d\mathcal{H}_\infty^\beta \right)^{\frac{1}{s}} \to 0,
\end{equation*}
as $k\to \infty$. Applying Theorem \ref{HHestimate},
\begin{align*}
    \left( \int_{\R^n} |I_\alpha (b_k f)(x) - I_\alpha (bf)(x)|^r  \, d\mathcal{H}_\infty^\beta \right)^{\frac{1}{r}} = & \left( \int_{\R^n} |I_\alpha (b_k f- bf)(x)|^r  \, d\mathcal{H}_\infty^\beta \right)^{\frac{1}{r}}\\
    \le & C \left( \int_{\R^n} |b_k(x)f(x) - b(x)f(x)|^s  \, d\mathcal{H}_\infty^\beta \right)^{\frac{1}{s}}\to 0,
\end{align*}
as $k\to \infty$, where $r$ is the corresponding exponent given by Theorem \ref{HHestimate}. This norm convergence implies that there exists a subsequence \(\{b_{k_j}\}_{j\in\mathbb{N}}\) such that
\(b_{k_j}\to b\) and $I_\alpha (b_{k_j} f ) \to I_\alpha ( b f )$ $\mathcal{H}^{\beta}_{\infty}$-a.e. (see, for example, \cite[Proposition 2.5]{PS_2023}). Consequently, \([b_{k_j},I_\alpha]f(x) \to [b,I_\alpha]f(x)\) for \(\mathcal{H}^{\beta}_{\infty}\)-a.e. $x\in \R^n$. Therefore, by Fatou's lemma, we conclude
\begin{align*}
\left(\int_{\mathbb{R}^n} |[b,I_\alpha]f|^{q}\, d\mathcal{H}^{\beta}_{\infty}\right)^{\frac{1}{q}}
&=
\left(\int_{\mathbb{R}^n} \lim_{j\to\infty} \Bigl|[b_{k_j},I_\alpha]f\Bigr|^{q}\,
d\mathcal{H}^{\beta}_{\infty}\right)^{\frac{1}{q}} \\
&\le
\liminf_{j\to\infty}
\left(\int_{\mathbb{R}^n} |[b_{k_j},I_\alpha]f|^{q}\, d\mathcal{H}^{\beta}_{\infty}\right)^{\frac{1}{q}} \\
&\le
C\,\liminf_{j\to\infty}\|b\|_{\mathrm{BMO}^{\beta}(\mathbb{R}^n)}
\left(\int_{\mathbb{R}^n} |f|^{p}\, d\mathcal{H}^{\beta}_{\infty}\right)^{\frac{1}{p}} \\
&=
C\,\|b\|_{\mathrm{BMO}^{\beta}(\mathbb{R}^n)}
\left(\int_{\mathbb{R}^n} |f|^{p}\, d\mathcal{H}^{\beta}_{\infty}\right)^{\frac{1}{p}}.
\end{align*}
\end{proof}

We now turn to the converse implication. In this part, we show that the boundedness of the commutator forces the symbol $b$ to belong to $\operatorname{BMO}^\beta$. This establishes the necessity of the condition and completes the characterization.

\begin{theorem}\label{Thm b BMO}
	Let $0<\beta \le n$, $0<\alpha <\beta$ and let $b\in L^1_{loc}(\R^n)$. For each $1<p<\frac{\beta}{\alpha}$, we define the exponent $q$ by
	\begin{equation*}
		\frac{1}{p}-\frac{1}{q}= \frac{\alpha}{\beta}.
	\end{equation*}
	
	If the commutator $[b,I_\alpha]$ satisfies the inequality
	\begin{equation}\label{LpLq bound commutator 2}
		\left( \int_{\R^n} \left| [b, I_\alpha] f \right|^q d\mathcal{H}_\infty^\beta \right)^\frac{1}{q} \le C  \left( \int_{\R^n} \left|f \right|^p d\mathcal{H}_\infty^\beta \right)^\frac{1}{p},
	\end{equation}
	for some $C>0$ and for every bounded measurable function $f
    $ with compact support, then $b\in \operatorname{BMO}^\beta (\R^n)$. Specifically, if we denote by $\| [b,I_\alpha]\|_{(p,q)}$ denote the best possible constant in the inequality \eqref{LpLq bound commutator 2}, there exists constant $C=C(\alpha,\beta, n)>0$ such that $\|b\|_{\operatorname{BMO}^\beta (\R^n)} \le C \| [b,I_\alpha]\|_{(p,q)}$.
\end{theorem}

The proof follows the lines of the proof of Lemma 5.1 in \cite{Holmes}, with slight modifications to adapt it to our nonlinear context. We also refer to \cite{Janson} and \cite{Chaffee}, where similar proofs are done for Calderon-Zygmund singular integrals in the scalar and bilinear case, respectively. 

\begin{proof}
    Fix a cube $Q$. Let us consider the cube $P$ with $\ell(P)=4\ell(Q)$, with the same bottom left corner. Let $P_R$ be the upper right half of $P$. We observe that $\ell(P_R)=2\ell(Q)$. Given $x\in Q$ and $y\in P_R$, we have the following trivial estimates, 
    \begin{align*}
        \frac{|x-y|}{2\sqrt{n} \ell(P)}\ge \frac{\sqrt{n}\ell(Q)}{2\sqrt{n}\ell(P)}=\frac{1}{8},
    \end{align*}
    and,
    \begin{align*}
        \frac{|x-y|}{2\sqrt{n} \ell(P)}\le  \frac{\sqrt{n}\ell(P)}{2\sqrt{n}\ell(P)}=\frac{1}{2}.
    \end{align*}
    There exist a function $K(x)$, that is smooth on $[-1,1]^n$, has a smoth periodic extension to $\R^n$, and is equal to $|x|^{n-\alpha}$ for $\frac{1}{8}\le |x|\le \frac{1}{2}$. Therefore, for $x\in Q$ and $y\in P_R$, we have
    \begin{equation*}
        \left( \frac{|x-y|}{2\sqrt{n}\ell(P)}\right)^{n-\alpha} = K \left(\frac{x-y}{2\sqrt{n}\ell(P)}\right).
    \end{equation*}
    We can express the function $K$ as an absolutely convergent Fourier series. 

    Now, we define $c_Q=\frac{1}{|P_R|}\int_{P_R} b(y)dy$ and $\sigma(x)=\sgn (b(x)-c_Q)$. Then, we have
    \begin{align*}
        \int_Q |b-c_Q| d\mathcal{H}_\infty^\beta = & \int_Q \left| b(x)- \frac{1}{|P_R|}\int_{P_R} b(y)dy\right| d\mathcal{H}_\infty^\beta(x)\\
        = & \frac{1}{\ell(P_R)^n}\int_Q \left| \int_{P_R} (b(x)-  b(y)) \sigma(x) dy\right| d\mathcal{H}_\infty^\beta(x)\\
        = & \frac{1}{\ell(P_R)^n}\int_Q \left| \int_{P_R} \frac{b(x)-  b(y)}{\left( \frac{|x-y|}{2\sqrt{n} \ell(P)} \right)^{n-\alpha} } \left( \frac{|x-y|}{2\sqrt{n} \ell(P)} \right)^{n-\alpha} \sigma(x) dy\right| d\mathcal{H}_\infty^\beta(x)\\
        = & C_{n,\alpha}  \frac{1}{\ell(P)^\alpha }\int_Q \left| \int_{P_R} \frac{b(x)-  b(y)}{ |x-y|^{n-\alpha} } K \left( \frac{x-y}{2\sqrt{n} \ell(P)} \right)\sigma(x) dy\right| d\mathcal{H}_\infty^\beta(x)\\
        = & I_1
    \end{align*}
    We expand $K$ as an absolutely convergent Fourier series, 
    \begin{equation*}
        K \left( \frac{x-y}{2\sqrt{n} \ell(P)} \right) = \sum_k a_k e^{ik \frac{x}{2\sqrt{n} \ell(P)}}e^{-ik \frac{y}{2\sqrt{n} \ell(P)}},
    \end{equation*}
    where we remark that $\sum_k |a_k|=C_{\alpha, n} <\infty$. Using subadditivity of the Choquet integral with respect to $\cal{H}_\infty^\beta$ we have
    \begin{align*}
        I_1 \le & C_{n,\alpha,\beta }  \frac{1}{\ell(P)^\alpha } \sum_k |a_k| \int_Q \left| \int_{P_R} \frac{b(x)-  b(y)}{ |x-y|^{n-\alpha} }  e^{-ik \frac{y}{2\sqrt{n} \ell(P)}} e^{ik \frac{x}{2\sqrt{n} \ell(P)}} \sigma(x) dy\right| d\mathcal{H}_\infty^\beta(x)\\
        = & C_{n,\alpha,\beta }  \frac{1}{\ell(P)^\alpha } \sum_k |a_k| \int_Q \left| \int_{\R^n} \frac{b(x)-  b(y)}{ |x-y|^{n-\alpha} }  e^{-ik \frac{y}{2\sqrt{n} \ell(P)}} \chi_{P_R}(y)   dy\,  e^{ik \frac{x}{2\sqrt{n} \ell(P)}} \sigma(x)\right| d\mathcal{H}_\infty^\beta(x)\\
        = & C_{n,\alpha,\beta }  \frac{1}{\ell(P)^\alpha } \sum_k |a_k| \int_Q \left| [b, I_\alpha] f_k(x)  e^{ik \frac{x}{2\sqrt{n} \ell(P)}} \sigma(x)\right| d\mathcal{H}_\infty^\beta(x)\\
        = & C_{n,\alpha,\beta }  \frac{1}{\ell(P)^\alpha } \sum_k |a_k| \int_Q \left| [b, I_\alpha] f_k(x) \right| d\mathcal{H}_\infty^\beta(x)\\
        = &I_2
    \end{align*}
    where $f_k(y)=e^{-ik \frac{y}{2\sqrt{n} \ell(P)}} \chi_{P_R}(y)$. We can estimate each integral using Hölder's inequality and \eqref{LpLq bound commutator 2},
    \begin{align*}
        \int_Q \left| [b, I_\alpha] f_k(x) \right| d\mathcal{H}_\infty^\beta(x) = & \int_{\R^n} \left| [b, I_\alpha] f_k(x) \right| \chi_Q(x) d\mathcal{H}_\infty^\beta(x)\\
        \le & 2\left( \int_Q \left| [b, I_\alpha] f_k(x) \right|^q d\mathcal{H}_\infty^\beta(x) \right) ^\frac{1}{q} \ell(Q)^\frac{\beta}{q'}\\
        \le & 2 \| [b,I_\alpha]\|_{(p,q)} \left( \int_Q \left| f_k(x) \right|^p d\mathcal{H}_\infty^\beta(x) \right) ^\frac{1}{p} \ell(Q)^\frac{\beta}{q'}\\
        = & 2 \| [b,I_\alpha]\|_{(p,q)} \ell(P_R)^\frac{\beta}{p} \ell(Q)^\frac{\beta}{q'}\\
        = & 2^{1+\frac{\beta}{p}} \| [b,I_\alpha]\|_{(p,q)} \ell(Q)^{\frac{\beta}{p} +\frac{\beta}{q'}}\\
        = & 2^{1+\frac{\beta}{p}} \| [b,I_\alpha]\|_{(p,q)} \ell(Q)^{\alpha +\beta },
    \end{align*}
    where we have used the relation $\frac{1}{q'}+\frac{1}{p}= \frac{\alpha}{\beta}+1$. Then, we have
    \begin{align*}
        I_2 \le & C_{n,\alpha,\beta }  \frac{1}{\ell(P)^\alpha } \sum_k |a_k|  2^{1+\frac{\beta}{p}} \| [b,I_\alpha]\|_{(p,q)} \ell(Q)^{\alpha +\beta } \\ 
        \le & C_{n,\alpha,\beta }  \frac{1}{\ell(Q)^\alpha } \| [b,I_\alpha]\|_{(p,q)} \ell(Q)^{\alpha +\beta } \sum_k |a_k| \\
        = & C_{n,\alpha,\beta }   \| [b,I_\alpha]\|_{(p,q)} \ell(Q)^{\beta } .
    \end{align*}
    Therefore, we have proved that for each cube $Q$ there exist $c_Q\in\R$ such that 
    \begin{equation*}
        \frac{1}{\ell(Q)^\beta}  \int_Q |b-c_Q| d\mathcal{H}_\infty^\beta  \le C_{n,\alpha,\beta }   \| [b,I_\alpha]\|_{(p,q)},
    \end{equation*}
    for all cube $Q$, and this implies $\|b\|_{\operatorname{BMO}^\beta(\R^n)} \le C_{n,\alpha,\beta }   \| [b,I_\alpha]\|_{(p,q)}$.
\end{proof}

This completes the proof of Theorem \ref{thm: characterization BMO commutator}, upon combining Theorems \ref{Thm LpLq bound commutator} and \ref{Thm b BMO}.

\section{Endpoint modular weak-type inequality}

In this section, we establish the endpoint modular weak-type inequality stated in Theorem \ref{thm: modular weak}. Using the pointwise estimate \eqref{key pointwise estimate 2} and the capacitary Fefferman-Stein inequality stated in Theorem \ref{thm: beta FS}, we combine these ingredients with the boundedness properties of the Riesz potential and the fractional maximal operator to obtain the desired endpoint control. This result extends the modular estimate of Cruz-Uribe and Fiorenza to the capacitary setting.

\begin{proof}[Proof of Theorem \ref{thm: modular weak}]
    The proof follows the lines of the proof of Theorem 1.2 in \cite{Cruz-Uribe-Fiorenza}. We refer to that paper for more details. 

    We may assume that $\|b\|_{\operatorname{BMO}^\beta(\R^n)}\neq 0$ and that $f$ is non-negative. By homogeneity it is enough to prove \eqref{eq: modular weak} when $t=1$. We have
    \begin{align*}
        \mathcal{H}_\infty^\beta (\{x :  |[b, I_\alpha] f(x)|>1\}) = & \Psi(B(1)) \frac{1}{\Psi(B(1))} \mathcal{H}_\infty^\beta (\{x :  |[b, I_\alpha] f(x)|>1\})\\
        \le & \Psi(B(1)) \sup_{t>0} \frac{1}{\Psi(B(1/t))} \mathcal{H}_\infty^\beta (\{x :  |[b, I_\alpha] f(x)|>t\})\\
        \le & \Psi(B(1)) \sup_{t>0} \frac{1}{\Psi(B(1/t))} \mathcal{H}_\infty^\beta (\{x :  \mathcal{M}_{\mathcal{H}_\infty^{\beta, Q_0}}([b, I_\alpha] f)(x)>t\}).
    \end{align*}
    Since $\varphi(t)=\frac{1}{\Psi(B(1/t))}$ is doubling \cite[p. 123]{Cruz-Uribe-Fiorenza}, we can apply the modular Fefferman-Stein inequality \eqref{weak orlicz FS}. Thus, by \eqref{weak orlicz FS} and Theorem \ref{thm: key pointwise}, we have
    \begin{align*}
        \mathcal{H}_\infty^\beta (\{x :  |[b, I_\alpha] f(x)|>1\}) \le & C \sup_{t>0} \frac{1}{\Psi(B(1/t))} \mathcal{H}_\infty^\beta (\{x :  \mathcal{M}_{\beta}^\# ([b, I_\alpha] f)(x)>t\})\\
        \le & C \sup_{t>0} \frac{1}{\Psi(B(1/t))} \mathcal{H}_\infty^\beta \left( \left \lbrace x :  I_\alpha f(x)  > \frac{t}{C \|b\|_{\operatorname{BMO}^\beta }} \right\rbrace \right)\\
        & + C \sup_{t>0} \frac{1}{\Psi(B(1/t))} \mathcal{H}_\infty^\beta \left( \left \lbrace x :  \mathcal{M}_{\alpha, B, \mathcal{H}_\infty^\beta} f(x)  > \frac{t}{C \|b\|_{\operatorname{BMO}^\beta }} \right\rbrace \right)\\
        = & I_1 + I_2.
    \end{align*}

    We need to use the weak-type estimate for the Riesz potential in the capacitary context to estimate $I$. More precisely, we will use Theorem \ref{HHestimate} in the case $p=1$, 
    \begin{equation*}
        \left( \int_{\R^n} |I_\alpha f(x)|^{\frac{\beta}{\beta -\alpha }}  d\mathcal{H}_\infty^\beta \right)^\frac{\beta-\alpha}{\beta } \le C \int_{\R^n} |f(x)| d\mathcal{H}_\infty^\beta.
    \end{equation*}
    Note that for $0<\beta<n$ the previous inequality is not the endpoint, $p=\frac{\beta}{n}$. However, for each $0<\beta\le n$ we have the following weak-type inequality,
    \begin{equation*}
        \mathcal{H}_\infty^\beta (\{ x\in \R^n : |I_\alpha f(x)|>t\}) \le C \left( \frac{1}{t} \left(\int_{\R^n} |f(x)| d\mathcal{H}_\infty^\beta \right)\right)^{\frac{\beta}{\beta-\alpha}},
    \end{equation*}
    for every $t>0$. Using the previous estimate, we have
    \begin{align*}
        I_1 \le & C \sup_{t>0} \frac{1}{\Psi(B(1/t))} \left( \frac{ \|b\|_{\operatorname{BMO}^\beta }}{t} \int_{\R^n} |f(x)| d\mathcal{H}_\infty^\beta \right)^{\frac{\beta}{\beta-\alpha}}\\
        = & C\sup_{t>0} \frac{1}{\Psi(B(1/t))} \frac{1}{t^{\beta/(\beta-\alpha)}}   \left(\int_{\R^n} \left( \|b\|_{\operatorname{BMO}^\beta } |f(x)| \right) d\mathcal{H}_\infty^\beta \right)^{\frac{\beta}{\beta-\alpha}}\\
        \le & C \sup_{t>0} \frac{1}{\Psi(B(1/t))} \frac{1}{t^{\beta/(\beta-\alpha)}}   \left(\int_{\R^n} B \left( \|b\|_{\operatorname{BMO}^\beta } |f(x)| \right)   d\mathcal{H}_\infty^\beta \right)^{ \frac{\beta}{\beta-\alpha}}\\
        \le & C  \Psi \left( \int_{\R^n} B \left( \|b\|_{\operatorname{BMO}^\beta } |f(x)| \right) d\mathcal{H}_\infty^\beta\right) 
    \end{align*}
    since $t \le B(t)$ and $t^{\beta /(\beta-\alpha)}\le \Psi(t)$ for all $t>0$, and in the last inequality we have used 
    \begin{equation*}
        \sup_{t>0} \frac{1}{\Psi(B(1/t))} \frac{1}{t^{\beta/(\beta-\alpha)}} \le C,
    \end{equation*}
    we refer to the reader to \cite[p. 124]{Cruz-Uribe-Fiorenza} for more details.
    
    To estimate $I_2$, we use the weak-type estimate for the capacitary Orlicz maximal function $\mathcal{M}_{\alpha, B, \mathcal{H}_\infty^\beta}$ established in Theorem \ref{thm: weak orlicz fractional}, obtaining
    \begin{align*}
        I_2 \le & C \sup_{t>0} \frac{1}{\Psi(B(1/t))} \Psi \left( \int_{\R^n} B\left(\frac{\|b\|_{\operatorname{BMO}^\beta} |f(x)|}{t}  \right) d\mathcal{H}_\infty^\beta \right)\\
        \le & C \sup_{t>0} \frac{1}{\Psi(B(1/t))} \Psi(B(1/t)) \Psi \left( \int_{\R^n} B\left(\|b\|_{\operatorname{BMO}^\beta} |f(x)|  \right) d\mathcal{H}_\infty^\beta \right)\\
        \le & C \Psi \left( \int_{\R^n} B\left(\|b\|_{\operatorname{BMO}^\beta} |f(x)|  \right) d\mathcal{H}_\infty^\beta \right),
    \end{align*}
    where we have used that $B$ and $\Psi$ are submultiplicative functions. This concludes the proof. 
\end{proof}

\section*{Acknowledgment}
The authors are grateful to David Cruz-Uribe for suggesting this problem to the first author during his visit to National Taiwan Normal University (Taipei, Taiwan), and for his valuable comments and suggestions on an earlier draft of this manuscript.

\section*{Conflicts of Interest}
The authors have no conflicts of interest to declare.

\section*{Data Availability Statement}
Data sharing is not applicable to this article, as no datasets were generated or analyzed during the current study.

\end{document}